\documentclass{article}
\usepackage{amsmath,amsfonts,amssymb}

\usepackage{hyperref}
\hypersetup{colorlinks,citecolor=blue,urlcolor=blue}
\RequirePackage{graphicx}

\newtheorem{theorem}{Theorem}

\newtheorem{definition}[theorem]{Definition}

\newtheorem{lemma}[theorem]{Lemma}

\newtheorem{proposition}[theorem]{Proposition}
\newtheorem{remark}[theorem]{Remark}

\newenvironment{proof}[1][Proof]{\noindent\textbf{#1.} }{\ \rule{0.5em}{0.5em}}

\begin{document}

\title{\textbf{Relaxed Optimal Control Problem for a Finite Horizon G-SDE with Delay and Its Application in Economics}}
\author{\textbf{Omar Kebiri$^{1,2}$ and Nabil Elgroud$^{3,1}$}\\ 
$^{1}$Institute of Mathematics, Brandenburgische Technische Universit\"at\\ Cottbus-Senftenberg, Cottbus, Germany. \\
$^{2}$Institute of Mathematics, The free University of Berlin, Berlin, Germany.\\
$^{3}$Lab. of Probability and Statistics (LaPS), Department of \\Mathematics, Badji Mokhtar University, Annaba, Algeria.}
\date{}
\maketitle
\begin{abstract}

This paper investigates the existence of a G-relaxed optimal control of a controlled stochastic differential delay equation driven by G-Brownian motion (G-SDDE in short). First, we show that optimal control of G-SDDE exists for the finite horizon case. We present as an application of our result an economic model, which is represented by a G-SDDE, where we studied the optimization of this model. We connected the corresponding Hamilton Jacobi Bellman equation of our controlled system to a decoupled G-forward backward stochastic differential delay equation (G-FBSDDE in short). Finally, we simulate this G-FBSDDE to get the optimal strategy and cost. 

\end{abstract} 
 


\section{Introduction} \label{sec:1}

The classical stochastic optimal control problem for delayed systems has
received a lot of attention. Systems where the dynamics are
influenced by the current value of the state,  as well as the past values, are called stochastic differential delay equations (SDDEs).
This type of model is useful in situations where there is some memory in the
dynamics, such as economics and finance, as well as the population growth
models in biology (see \cite{Ivanov2008}, \cite{mansour2022}, \cite{stoica2004}).
Elsanosi et al \cite{Elsanosi2000} studied the stochastic maximum principle with delay and proved a verification theorem of a variational inequality. In \cite{Ivanov2008}, Ivanov studied the optimal control of stochastic
differential delay equations and gave an application to a stochastic model
in economics. Menoukeu-Pamen \cite{Menoukeu2015} obtained a necessary and sufficient
condition of optimality of stochastic maximum principle for delayed
stochastic differential games for a general non-Markovian stochastic
control problem under model uncertainty and delay. In \cite{oksendal2011}
the authors studied sufficient and necessary stochastic maximum principles of
time-delayed stochastic differential equations with jumps. Agram et al \cite{Agram2014} derived stochastic maximum principles for optimal control under
partial information in the infinite horizon for a system governed by
forward-backward stochastic differential equations (FBSDEs in short) with
delay. The existence of optimal control of non-linear
multiple-delay systems having an implicit derivative with quadratic
performance criteria by suitably adopting some of the techniques treated by 
\cite{balachandran1989}. Rosenblueth \cite{rosenblueth1992} proposed two different
relaxation procedures for optimal control problems involving transformations
of the state, and the control functions show that the resulting relaxed
problem has a solution, for which the existence of a minimizer is
assured. The optimal control of systems is described by delay differential
equations with a quadratic cost functional (see $e.g.,$ \cite{patel1982}).
Motivated by a concept of uncertainty which appears in many areas of sciences
contains inaccurate parameters, in general, the liquidity in the markets, risks
resulting from dark fluctuations, and their impact on the movement of the asset
prices and financial crises. In particular, the optimal portfolio choice
problems where the risk premium processes and the volatility are unknown.

Peng \cite{peng2007,peng2010}involved the sublinear
or G-expectation space with a process called G-Brownian motion, also
constructed It\^{o}'s stochastic calculus with respect to the G-Brownian
motion, and Gao \cite{gao2009} proved the existence and uniqueness of the solution of stochastic
differential equations driven by G-Brownian motion (G-SDEs). Furthermore, Fei et al \cite{Fei2019} studied the existence and stability of solutions
to highly nonlinear stochastic differential delay equations driven by G-Brownian motion (G-\text{SDDEs}). Some properties of
numerical solutions for semilinear stochastic delay differential equations
driven by G-Brownian motion have been studied by \cite{Yua2021}.

Xu in \cite{xu2010} studied the existence and uniqueness  theorem of backward
stochastic differential equations (BSDEs) under super linear expectation to provide probabilistic interpretation for the viscosity
solution of a class of Hamilton-Jacobi-Bellman (HJB in short) equations, he also shows that
BSDEs under super linear expectation could characterize a class of
stochastic control problems. The authors in  \cite{Hartmann2013} used the HJB equations that are recognized as the dynamic
programming equations of the optimal control problems.

Recently, the formulation of optimal control\ problem in a G-Brownian
motion was studied by \cite{biagini2018,hu2018,hu2014,Yin2022,Ren2017,Ren2018}. Moreover in \cite{redjil2018} they also study the
problem of the existence of optimal relaxed control for
stochastic differential equations driven by a G-Brownian motion. Moreover in \cite{Elgroud2022} the authors  studied the existence of relaxed optimal control for G%
-neutral stochastic functional differential equations with simulations results based on \cite{yang2016}.

In this paper, we consider an optimal control of systems governed by a
stochastic differential delay equation driven by a G-Brownian motion (G-SDDE)
\begin{eqnarray}
\left\{ 
\begin{aligned}
dX\left( t\right) &=b\left( t,X\left( t\right) ,X\left( t-\tau \right),u\left( t\right) ,u\left( t-\tau \right) \right) dt  \\ 
&+\gamma \left( t,X\left( t\right) ,X\left( t-\tau
\right) ,u\left( t\right) ,u\left( t-\tau \right) \right) d\left\langle B\right\rangle _{t} \\ 
&+\sigma \left( t,X\left( t\right) ,X\left( t-\tau
\right) \right) dB_{t},\quad t\in \left[ 0,T\right], \\ 
X\left( t\right) &=\eta(t),\quad t\in \left[ -\tau ,0\right] , \tau >0,
\end{aligned}
\right.\label{1.1}
\end{eqnarray}
 where $\eta$ $=\left\{ \eta \left( \theta \right) \right\} _{-\tau
\leq \theta \leq 0}$ $\in C\left( \left[ -\tau ,0\right] ;\mathbb{R}
^{n}\right)$, $0<\tau< T$, and $(B_{t})_{t\geq 0},$ is a $m$-dimensional 
G-Brownian motion defined on a space of sublinear expectation $\left(
\Omega ,\mathcal{H},\widehat{E}\right)$, with a universal filtration $
\mathbb{F}^{\mathcal{P}}=\left\{ \widehat{\mathcal{F}}_{t}^{\mathcal{P}
}\right\} _{t\geq 0}$, and using the space of probability measures $\mathcal{P}\left( \mathbb{A}\right) $
on $\mathcal{B}(\mathbb{A})$, where $\mathcal{B}(\mathbb{A})$ is the $\sigma $
-algebra of Borel subsets of the set $\mathbb{A}$ of values taken by the strict controls. $u\in \mathcal{U}\left( \left[ 0,T\right] \right)$ is called a strict control, with value in the action space $\mathbb{A}$, where we denote $\mathcal{U=U}\left( \left[ 0,T\right] \right)$ is a set of strict controls, and $\mathbb{A}$ is a compact polish subspace of $\mathbb{R}^{n}$.
$\left( \left\langle B\right\rangle _{t}\right) _{t\geq
0}$ is a quadratic variation process of G-Brownian motion.

Within the G-Brownian motion framework studied in \cite{peng2007,peng2010}, the controlled system $\left( \ref{1.1}\right) 
$ that minimizes the cost functional with the finite horizon $\left(
T<\infty \right) $ in which the coefficient of the initial cost depends on not
only on the current value of the state but also on past value, which is given by
\begin{eqnarray}
\widehat{E}\left[ \int_{0}^{T}\mathcal{L}(t,X\left( t\right) ,X\left(
t-\tau \right) ,u\left( t\right) ,u\left( t-\tau \right) )\,dt+\Psi (X\left(T\right) )\right]   \nonumber\\
=\underset{\mathbb{P\in}\mathcal{P}}{\sup}E^{\mathbb{P}}\left[\int_{0}^{T}\mathcal{L}(t,X\left(t\right) ,X\left(t-\tau\right),u\left(t\right),u\left(t-\tau\right))\,dt+\Psi (X\left(T\right))\right].            \label{1.2}
\end{eqnarray}

Our results are based on the aggregation property and the tightness
arguments of the distribution of the control problem.

This paper is organized as follows: In section \thinspace \ref{sec:2}, we
recall some notions and preliminaries, and we formulate our problem. In section \thinspace \ref{sec:3}, first, we introduce the
space of relaxed control and the problem of G-relaxed control of G-SDDE.
In section \thinspace \ref{sec:4}, we established the existence of a minimizer of the cost function
for the finite horizon. Then, we prove our main result, which is the existence of G-relaxed optimal control. In the last section, we provide a stochastic model in economics and its
optimization that have delay and randomness in the production cycle, In the last section, where the noise of the system is big, which prevents us to estimate the noise parameter. This will lead to a G-SDDE.
To minimize the
investment capital under the assumptions of labor, we obtain a probabilistic representation of the solution to the HJB equation, expressed in the form of a system of decoupled forward backward stochastic differential delay equations driven by G-Brownian motion. In the end, we present the simulation result of this G-FBSDDE.
 \section{Preliminaries and formulation of problem}\label{sec:2}

This section aims to give some basic concepts and results
of G-stochastic calculus. More details about these are included in \cite{denis2006,denis2011,soner2011m,soner2011,peng2007,peng2010}, and we introduce the formulation of the problem of optimal control for n-dimensional
stochastic differential delay equation is driven by G-Brownian motion (G-SDDE).

Let $\Omega :=\{\omega \in C(\left[ 0,T\right] ;\mathbb{R}^{n}):\omega
(0)=0\}$, be the space of real valued continuous functions on $\left[ 0,T%
\right] $ such that $\omega (0)=0,$\ equipped with the following distance
\begin{equation*}
d\left( w^{1},w^{2}\right) :=\sum_{N=1}^{\infty }2^{-N}\left( \left( 
\underset{0\leq t\leq N}{\max }\left\vert w_{t}^{1}-w_{t}^{2}\right\vert
\right) \wedge 1\right),
\end{equation*}
and $\Omega _{t}:=\{w_{.\wedge t}:w\in \Omega \}$, $B_{t}\left( w\right)
=w_{t},t\geq 0$ is the canonical process on $\Omega ,$ and let $\mathbb{F}:=(%
\mathcal{F}_{t}\,)_{\,t\geq 0}$ be the natural filtration generated by $%
(B_{t}\,)_{\,t\geq 0}$. In addition, for each $t\in \left[ 0,\infty \right), $we set 
\begin{eqnarray*}
\mathcal{F}_{t^{+}} &:&=\cap _{s>t}\mathcal{F}_{s}, \\
\mathbb{F^{+}} &:&=(\mathcal{F}_{t^{+}}\,)_{\,t\geq 0}, \\
\mathcal{F}_{t}^{\mathbb{P}} &:&=\mathcal{F}_{t^{+}}\vee \mathcal{N}^{
\mathbb{P}}(\mathcal{F}_{t^{+}}), \\
\widehat{\mathcal{F}}_{t}^{\mathbb{P}} &:&=\mathcal{F}_{t^{+}}\vee \mathcal{N}^{
\mathbb{P}}(\mathcal{F}_{\infty }),
\end{eqnarray*}
with $\mathcal{N}^{\mathbb{P}}(\mathcal{G})$ denotes a $\mathbb{P}$%
-negligible set on a $\sigma $-algebra $\mathcal{G}$ given by%
\begin{equation*}
\mathcal{N}^{\mathbb{P}}(\mathcal{G}):=\{D\subset \Omega :\exists \ 
\widetilde{D}\in \mathcal{G}\ \ \text{such\ that}\ D\subset \widetilde{D}\ 
\text{\ and}\ \mathbb{P}[\widetilde{D}]=0\}.
\end{equation*}
Consider the following space, for $t\in \left[ 0,T\right]$
\begin{eqnarray*}
Lip(\Omega _{t}) &:&=\left\{ \varphi (B_{t_{1}},...,B_{t_{n}}):\varphi \in
C_{b,Lip}(\mathbb{R}^{n})\text{ and }t_{1},t_{2},...,t_{n}\in \left[ 0,t
\right] \right\} , \\
Lip(\Omega ) &:&=\underset{n\in \mathbb{N}}{\cup }Lip(\Omega _{n}),
\end{eqnarray*}
with $C_{b,Lip}(\mathbb{R}^{n})$ denotes a space of bounded and Lipschitz on 
$\mathbb{R}^{n}.$ Let $T>0$ be a fixed time.
We have the following lemma
\begin{lemma}
For any $\mathbb{P}$ arbitrary probability measure on $\left( \Omega,\mathcal{F}_{\infty }\right).$ There is a $\mathbb{P}$-a.s. unique $\mathcal{
F}_{t}$-measurable random variable for each $\xi $ such that, for every $
\widehat{\mathcal{F}}_{t}^{\mathbb{P}}$-measurable random variable $\widehat{%
\xi }$, we have 
\begin{equation*}
\xi =\widehat{\xi },\quad\mathbb{P}-a.s.
\end{equation*}
\end{lemma}

For every $\widehat{\mathbb{F}}\mathbb{^{\mathbb{P}}}$-progressively
measurable process $\widehat{X}$, there is a unique $\mathbb{F}$-progressively measurable process $X$ such that
\begin{equation*}
X=\widehat{X},\quad dt\times\mathbb{P}-a.s.
\end{equation*}
Furthermore, if $\widehat{X}$ is $\mathbb{P}$-$a.s.$ continuous, then $X$ can
be chosen to $\mathbb{P}$-$a.s.$ continuous. 
To follow the details see (Lemma 2.1, \cite{soner2011m}).\\
The G-expectation $\widehat{E}$ $:\mathcal{H}:=Lip(\Omega
_{T})\longrightarrow \mathbb{R}$, constructed by \cite{peng2007}, is a
consistent sublinear expectation on the lattice $\mathcal{H}$ of real
functions, it satisfies:
\begin{enumerate}
\item Sub-additivity: $\widehat{E}[X+Y]\leq \widehat{E}[X]+\widehat{E}[Y],$\quad
for all $X,Y\in \mathcal{H},$
\item Monotonicity: $X\geq Y\Rightarrow \widehat{E}[X]\geq \widehat{E}[Y],$\quad
for all $X,Y\in \mathcal{H},$
\item Constant preserving: $\widehat{E}[c]=c,$\quad for all $c\in \mathbb{R},$
\item Positive homogeneity: $\widehat{E}[\lambda X]=\lambda \widehat{E}
[X],$\quad for all $\lambda \geq 0,$ $X\in \mathcal{H}.$
\end{enumerate}

If $1$ and $2$ are only satisfied, the triple $\left( \Omega, \mathcal{H},\widehat{E}\right) $ is said to be sub-linear expectation space, 
$\ \widehat{E}\left[ .\right] $ is also referred to as a nonlinear
expectation and the triple $\left( \Omega ,\mathcal{H},\widehat{E}\right)$
is called a nonlinear expectation space.
We suppose that, if $Y=(Y_{1},...,Y_{n}), Y_{i}$ $\in \mathcal{H}$, then $\varphi (Y_{1},...,Y_{n})\in \mathcal{H}$ for all $\varphi $ $\in C_{b,Lip}(
\mathbb{R}^{n}).$
\begin{definition}
Under $\widehat{E}$ if for any $\varphi \in C_{b,Lip}(\mathbb{R}^{n+m}),$ the
random vector $Y=(Y_{1},...,Y_{n})$ is said to be independent from another
random vector $X=(X_{1},...,X_{m}),$
\begin{equation*}
\widehat{E}\left[ \varphi \left( X,Y\right) \right] =\widehat{E}\left[ 
\widehat{E}\left[ \varphi \left( x,Y\right) \right] _{x=X}\right].
\end{equation*}
\end{definition}
\begin{definition}
Under the G-expectation $\widehat{E}[\cdot ],$ an $n$-dimensional random
vector $X$ on $\left( \Omega ,\mathcal{H},\widehat{E}\right) $ is said to be 
G-normally distributed for any $\varphi $ $\in C_{b,Lip}(\mathbb{R}^{n})$, if the function $u$ defined by%
\begin{equation*}
u(t,x):=\widehat{E}[\varphi (x+\sqrt{t}X)],\quad\left( t,x\right) \in \left[
0,T\right] \times 
\mathbb{R}^{n},
\end{equation*}
is the unique viscosity solution of the parabolic equation%
\begin{equation*}
\left\{ 
\begin{array}{l}
\frac{\partial u}{\partial t}-G(D^{2}u)=0,\quad (t,x)\in \left[ 0,T\right]\times \mathbb{R}^{n}, \\
u(0,x)=\varphi (x),
\end{array}
\right.
\end{equation*}
with $D^{2}u=(\partial _{x_{i}x_{j}}^{2}u)_{1\leq i,j\leq n}$ denotes the Hessian matrix of $u$ and let define the nonlinear operator G by 
\begin{equation*}
G(A):=\frac{1}{2}\sup_{\gamma \in \Gamma }\left\{ tr(\gamma \gamma ^{\ast
}A)\right\} ,\quad \gamma \in 
\mathbb{R}
^{n\times n},
\end{equation*}
where $\ A$ denotes a $n\times n$ symmetric matrix and $\Gamma $ is a
given non-empty, bounded, and closed subset of $%
\mathbb{R}
^{n\times n}$. The transpose of the vector $v$ is denoted by $v^{\ast }$. $\mathcal{N}(0,$ $\Sigma )$ denotes the G-normal distribution, where $%
\Sigma :=\{\gamma \gamma ^{\ast },\,\,\gamma \in \Gamma \}$.
\end{definition}
\begin{definition}[\textbf{G-Brownian Motion}]
The canonical process $\left( B_{t}\right) _{t\geq 0}$ on \\
$\left( \Omega ,%
\text{ }\mathcal{H},\text{{}}\widehat{E}\right) $ is called a G-Brownian
motion if the following properties hold:
\begin{itemize}
\item $B_{0}=0.$
\item \textit{For each} $t,s\geq 0$ \textit{the increment} $B_{t+s}-B_{t}$ \
is $\mathcal{N}(0,$ $s\Sigma )$-distributed$.$
\item $B_{t_{1}},B_{t_{2}},...,B_{t_{n}}$ is \textit{independent of} $B_{t},$
\textit{for\ }$n\geq 1$ \textit{and} $t_{1},t_{2},...,t_{n}\in \left[ 0,t\right].$
\end{itemize}
\end{definition}
For $p\geq 1,$ we use $L_{G}^{p}(\Omega _{T})$ to signify the completion of $%
Lip(\Omega _{T})$ under the natural norm%
\begin{equation*}
\Vert X\Vert _{L_{G}^{p}(\Omega _{T})}^{p}:=\widehat{E}[|X|^{p}],
\end{equation*}
and define $M_{G}^{0,p}(0,T)$ as the space of $\mathbb{F}$-progressively
measurable, $\mathbb{R}^{n}$-valued simple processes of the form%
\begin{equation*}
\eta (t)=\eta (t,w)=\sum_{i=0}^{n-1}\xi _{t_{i}}\left( w\right) \mathbb{I}%
_{[t_{i},t_{i+1})}(t),
\end{equation*}
with $\left\{ t_{0},\cdots ,t_{n}\right\} $ denotes a subdivision of $\left[0,T\right]$. The space $M_{G}^{p}(0,T)$ is called the closure of $M_{G}^{0,p}(0,T)$
with respect to the norm%
\begin{equation*}
\Vert \eta \Vert _{M_{G}^{p}(0,T)}^{p}:=\int_{0}^{T}\widehat{E}[
\underset{0\leq t\leq T}{\sup }|\eta (t)|^{p}ds].
\end{equation*}

Note that $M_{G}^{q}(0,T)\subset M_{G}^{p}(0,T)$ if $1\leq p<q$. For each $%
t\geq 0$, let $L^{0}(\Omega _{t})$ be the set of $F_{t}$-measurable
functions. We set%
\begin{equation*}
Lip(\Omega _{t}):=Lip(\Omega )\cap L^{0}(\Omega _{t}),\quad
L_{G}^{p}(\Omega_{t}):=L_{G}^{p}(\Omega )\cap L^{0}(\Omega _{t}).
\end{equation*}

For each $\eta \in M_{G}^{0,2}(0,T)$, the related It\^{o} integral of $%
\left( B_{t}\right) _{t\geq 0}$ is defined by%
\begin{equation*}
I(\eta )=\int_{0}^{T}\eta \left( s\right) d_{G}B_{s}:=\sum_{j=0}^{N-1}\eta
_{j}(B_{t_{j+1}}-B_{t_{j}}),
\end{equation*}

where $I:\,M_{G}^{0,2}(0,T)\rightarrow L_{G}^{2}(\Omega _{T})$ is the
mapping continuously extended to $M_{G}^{2}(0,T).$ The quadratic variation
process $\langle B\rangle _{t}^{G}$ of $\left( B_{t}\right) _{t\geq 0}$ is
not always a deterministic process, and can be formulated in $%
L_{G}^{2}(\Omega _{T})$ by the continuous $n\times n$%
-symmetric-matrix-valued process given by

\begin{equation}
\langle B\rangle _{t}^{G}:=B_{t}\otimes B_{t}-2\int_{0}^{t}B_{s}\otimes
d_{G}B_{s},  \label{2.1}
\end{equation}

where a diagonal is constituted of non-decreasing processes. Here, for $a,b\in $
$
\mathbb{R}
^{n}$, the $n\times n-$symmetric matrix $a\otimes b$ is defined by $a\otimes
bx$ $=(a\cdot x)b$ for $x\in 
\mathbb{R}
^{n}$, where \textquotedblright $\cdot $\textquotedblright\ represents the
scalar product in $\mathbb{R}^{n}$.

Let the mapping $\mathcal{J}_{0,T}\left( \eta \right)
:M_{G}^{0,1}(0,T)\mapsto \mathbb{L}_{G}^{1}(\Omega _{T})$ for each $\eta \in
M_{G}^{0,1}(0,T)$, \ given by
\begin{equation*}
\mathcal{J}_{0,T}\left( \eta \right) =\int_{0}^{T}\eta \left( t\right)
d\langle B\rangle _{t}^{G}:=\sum_{j=0}^{N-1}\xi _{j}(\langle B\rangle
_{t_{j+1}}^{G}-\langle B\rangle _{t_{j}}^{G}).
\end{equation*}

Then $\mathcal{J}_{0,T}\left( \eta \right) $ can be extended continuously to 
$\mathcal{J}_{0,T}\left( \eta \right) :M_{G}^{1}(0,T)\rightarrow \mathbb{L}%
_{G}^{1}(\Omega _{T})$ for each $\eta \in M_{G}^{0,1}(0,T),$%
\begin{equation*}
\mathcal{J}_{0,T}\left( \eta \right) :=\int_{0}^{T}\eta \left( t\right)
d\langle B\rangle _{t}^{G}.
\end{equation*}

We have the following properties (formulated for the case d = 1, for
simplicity).

\begin{lemma}
(\cite{peng2010}) For each $p\geq 1$, and $\eta \in M_{G}^{2}(0,T)$ we have 
\begin{eqnarray*}
\widehat{E}\left[ \left( \int_{0}^{T}\eta \left( t\right)d_{G}B_{t}\right) ^{2}\right] &=&
\widehat{E}\left[ \int_{0}^{T}\left(\eta \left(t\right)\right)^{2}d\langle B\rangle _{t}^{G}\right]\\
&\leq& \overline{\sigma }^{2}\widehat{E}\left[ \int_{0}^{T}\left(\eta \left(t\right)\right)^{2}dt
\right], \quad \left( \text{isometry}
\right),
\end{eqnarray*}
and, for each $\eta \in M_{G}^{p}(0,T)$, we have
\begin{eqnarray*}
\widehat{E}\left[ \int_{0}^{T}\left\vert \eta \left(t\right)\right\vert ^{p}dt\right]
&\leq& \int_{0}^{T}\widehat{E}\left[ \left\vert \eta \left(t\right)\right\vert ^{p}
\right] dt.
\end{eqnarray*}
\end{lemma}

\begin{proposition}
(\cite{denis2006}) Assume $d=1$. There exists a weakly compact family of
probability measures $\mathcal{P}$ on $(\Omega ,$ $\mathcal{B}(\Omega )),$ 
such that%
\begin{equation*}
\widehat{E}[\xi ]=\sup_{\mathbb{P}\in \mathcal{P}}E^{\mathbb{P}}[\xi ],\quad \text{
for each }\xi \in {\mathbb{L}}_{G}^{1}(\Omega ).
\end{equation*}
Then $c(.)$ is the associated regular Choquet capacity related to $\mathbb{P}
$ defined by
\begin{equation*}
c(C):=\sup_{\mathbb{P}\in \mathcal{P}}\mathbb{P}(C),\quad C\in \mathcal{B}%
(\Omega ).
\end{equation*}
\end{proposition}

\begin{definition}
\textit{If} $c(C)=0$\textit{ or equivalently if }$\mathbb{P}
(C)=0$ \textit{for all }$\mathbb{P}\in \mathcal{P}.$ \textit{A set }$C\in
B(\Omega )$ is polar if a property holds outside a polar set.
It is called quasi surely ($q.s.$).
\end{definition}

Let us define $\mathcal{N}_{\mathcal{P}}$ the $\mathcal{P}$-polar sets%
\begin{equation*}
\mathcal{N}_{\mathcal{P}}:=\bigcap_{\mathbb{P}\in \mathcal{P}}\mathcal{N}^{%
\mathbb{P}}(\mathcal{F}_{\infty }).
\end{equation*}

For the possibly mutually singular probability measures $\mathbb{P}$, $\mathbb{P
}\in \mathcal{P}$ in \cite{soner2011}, we must utilise the following
universal filtration $\mathbb{F}^{\mathcal{P}}$ 
\begin{eqnarray*}
\mathbb{F}^{\mathcal{P}} &:&=\{\widehat{\mathcal{F}}_{t}^{\mathcal{P}%
}\}_{t\geq 0},\\
\quad \widehat{\mathcal{F}}_{t}^{\mathcal{P}} &:&=\bigcap_{\mathbb{P}\in 
\mathcal{P}}(\mathcal{F}_{t}^{\mathbb{P}}\vee \mathcal{N}_{\mathcal{P}
}),\quad \text{for } t\geq 0.
\end{eqnarray*}

The dual expression of the G-expectation gives the following aggregation property.

\begin{lemma}
\label{aggreg}Let $\eta $ $\in 
M_{G}^{2}(0,T)$. Then, $\eta $ is It\^{o}-integrable for every $%
\mathbb{P}\in \mathcal{P}$. Moreover, for every $t\in \lbrack 0,T]$,%
\begin{equation*}
\int_{0}^{t}\eta(s)d_{G}B_{s}=\int_{0}^{t}\eta (s)dB_{s},\mathbb{P-}a.s.,\quad \text{for every }\mathbb{P}\in \mathcal{P}.
\end{equation*}
\end{lemma}
where the right-hand side is the standard It\^{o} integral. As a result, the quadratic variation process $\left\langle B\right\rangle _{t}^{G}$ defined in (\ref{2.1}) agrees with the standard quadratic variation process
quasi-surely. For more details see (Proposition 3.3, \cite{soner2011}). In the sequel, we will start by removing the notation G from both the G-stochastic integral and the G-quadratic variation.

Among the results of stability obtained in \cite{denis2011}, we quote the following which plays an essential part in our analysis.
\begin{lemma}\label{gh363}
If $\{\mathbb{P}%
_{n}\}_{n=1}^{\infty }$ $\subset $ $\mathcal{P}$ converges weakly to $
\mathbb{P}\in \mathcal{P}$. Then 
\begin{equation*}
E^{\mathbb{P}_{n}}[\xi ]\rightarrow E^{\mathbb{P}}[\xi ],\quad\text{for each }%
\xi \in \mathbb{L}_{G}^{1}(\Omega _{T}),
\end{equation*}
\end{lemma}
for more details see (Lemma 29, \cite{denis2011}).\\

In view of the dual formulation of the G-expectation, we end this section
by giving the following one-dimensional G-Burkholder-Davis-Gundy (G-BDG in short) type
estimates.
\begin{lemma}
\label{g-bdg1} (\cite{gao2009}) Assume d = 1. For each\textit{\ }$p\geq 2,
\eta \in M_{G}^{p}(0,T)$ and $0\leq s\leq t\leq T$, such that
\begin{equation*}
\widehat{E}\left[ \sup_{s\leq u\leq t}\left\vert \int_{s}^{u}\eta
(r)dB_{r}\right\vert ^{p}\right] \leq C_{p}|t-s|^{\frac{p}{2}-1}\int_{s}^{t}%
\widehat{E}[|\eta (r)|^{p}|]dr,
\end{equation*}
\end{lemma}

where $C_{p}>0$ is a constant depending only on $p$ and $T.$

\begin{lemma}
\label{g-bdg2} (\cite{gao2009}) \textit{For each }$p\geq 1$\textit{\ and }$%
\eta \in M_{G}^{p}(0,T),$\textit{\ then there exists a positive constant }$\overline{\sigma }$\textit{\ such that }$\frac{d\langle B\rangle _{t}}{dt}\leq 
\overline{\sigma }$\textit{\ }$q.s.,$\textit{ we have } 
\begin{equation*}
\widehat{E}\left[ \sup_{s\leq u\leq t}\left\vert \int_{s}^{u}\eta
(r)d\left\langle B\right\rangle _{r}\right\vert ^{p}\right] \leq \overline{%
\sigma }^{p}|t-s|^{p-1}\int_{s}^{t}\widehat{E}[|\eta (r)|^{p}]dr.
\end{equation*}
\end{lemma}

We study the existence of optimal control problem for $n$-dimensional
G-SDDE. We are concerned with the controlled system described by a G-SDDE, which
is given by
\begin{eqnarray}
\left\{ 
\begin{aligned}
dX\left( t\right)&=b\left( t,X\left( t\right) ,X\left( t-\tau \right),u\left( t\right) ,u\left( t-\tau \right) \right) dt \\
&+\gamma \left( t,X\left( t\right) ,X\left( t-\tau
\right) ,u\left( t\right) ,u\left( t-\tau \right) \right) d\left\langle B\right\rangle _{t} \\ 
&+\sigma \left( t,X\left( t\right) ,X\left( t-\tau
\right) \right) dB_{t},\quad t\in \left[ 0,T\right],\\ 
X\left( t\right) &=\eta \left( t \right),\quad t\in \left[ -\tau ,0\right] ,\tau >0,
\end{aligned}%
\right. \label{3.0}
\end{eqnarray}

where $\eta=\left\{ \eta \left( \theta \right) \right\}_{-\tau \leq \theta \leq 0}$ is $\widehat{\mathcal{F}}_{0}^{\mathcal{P}}-$adapted $
C\left( \left[ -\tau ,0\right] ;\mathbb{R}^{n}\right) $-value random variable, where $C=C([-\tau, 0]; \mathbb{R}^{n})$ is a space of continuous functions equipped with the norm $\left \vert \phi \right \vert_{C}=\underset{\theta\in [-\tau, 0]}{\sup}\left \vert \phi (\theta)\right \vert_{C}$,
as well as $\eta \in M_{\mathbb{G}}^{2}\ \left( \left[
-\tau ,0\right];\mathbb{R}^{n}\right).$ Then, the integral equation is given by
\begin{eqnarray}
X\left( t\right) &=&\eta \left( 0 \right)+\int_{0}^{t}b\left( s,X\left( s\right) ,X\left(
s-\tau \right) ,u\left( s\right) ,u\left( s-\tau \right) \right) ds \nonumber \\
&&+\int_{0}^{t}\gamma \left( s,X\left( s\right) ,X\left( s-\tau \right)
,u\left( s\right) ,u\left( s-\tau \right) \right) d\left\langle
B\right\rangle _{s}  \nonumber \\
&&+\int_{0}^{t}\sigma \left( s,X\left( s\right) ,X\left( s-\tau \right)
\right) dB_{s},  \label{3.1}
\end{eqnarray}

with $\left( \left\langle B\right\rangle _{t}\right) _{t\geq 0}$ is a quadratic
variation process of $m$-dimentional G-Brownian motion 
defined on a space of sublinear expectation $\left( \Omega ,\mathcal{H},
\widehat{E}\right) $, with an universal filtration $\mathbb{F}^{\mathcal{P}%
}=\left\{ \widehat{\mathcal{F}}_{t}^{\mathcal{P}}\right\} _{t\geq 0},$ and $u\left( t\right), u\left( t-\tau \right) $ $\in\mathbb{ A},$ are called
the strict control variables for each $t\in \left[ 0,T\right] .$ 
Moreover, the functions
\begin{eqnarray*}
b &:&\left[ 0,T\right] \times \mathbb{R}^{n}\times \mathbb{R}^{n}\times 
\mathbb{A}\times \mathbb{A} \rightarrow \mathbb{R}^{n}, \\
\gamma &:&\left[ 0,T\right] \times \mathbb{R}^{n}\times \mathbb{R}^{n}\times 
\mathbb{A}\times \mathbb{A} \rightarrow \mathbb{R}^{m\times n},\\
\sigma &:&\left[ 0,T\right] \times \mathbb{R}^{n}\times \mathbb{R}^{n} \rightarrow \mathbb{R}^{m\times n},
\end{eqnarray*}
as well as $\sigma \left( ., x, y\right),$ $b\left(., x, y, u\left( .\right), u\left( .-\tau\right) \right),$ $\gamma \left(., x, y, u\left( .\right), u\left( .-\tau\right) \right) \in \\
M_{G}^{2}(%
\left[ 0,T\right]; \mathbb{R}^{n})$ for each $x,y\in \mathbb{R}^{n}$ and for
each strict control $u$.

  \section{Relaxed control of the G-SDDE.}
\label{sec:3}The strict control problem may not have an optimal solution
in the absence of convexity assumptions because $\mathbb{A}$ is too small
to contain a minimizer. Then the strict control space must then be injected
into a larger space with good compactness and convexity properties. $
\mathcal{P}\left( \mathbb{A}\right) $ is the space of probability measures
on $\mathbb{A}$, endowed with its Borel $\sigma $-algebra $\mathcal{B}(%
\mathbb{A})$ whith $\mathbb{A}$ is the set of compact polish space.

For that we consider the class of G-relaxed stochastic controls on 
$(\Omega ,\mathcal{H},\widehat{E})$.
\begin{definition}[G-Relaxed stochastic control]
An $\mathbb{F}^{\mathcal{P}}$-progressively measurable random measure of the
form $q(\omega ,dt,d\xi )=\mu _{t}(\omega ,d\xi )dt$ is a G-relaxed
stochastic control on $(\Omega ,\mathcal{H},\widehat{E})$, such that$\ $%
\begin{eqnarray}
X\left( t\right) &=&\eta \left( 0 \right)+\int_{0}^{t}\int_{\mathbb{A}}b\left( s,X\left(
s\right) ,X\left( s-\tau \right) ,\xi _{1},\xi _{2}\right) \mu _{s}(d\xi
)ds \nonumber\\
&+&\int_{0}^{t}\int_{\mathbb{A}}\gamma \left( s,X\left(
s\right) ,X\left( s-\tau \right) ,\xi _{1},\xi _{2}\right) \mu _{s}(d\xi
)d\left\langle B\right\rangle _{s} \nonumber\\ 
&+&\int_{0}^{t}\sigma \left( s,X\left( s\right) ,X\left(
s-\tau \right) \right) dB_{s}. \label{3.3}
\end{eqnarray}
\end{definition}
where $(\xi_1,\xi_2)=\xi\in \mathbb{A}$.

It is important to note that, the set $\mathcal{U}([0,T])=\mathcal{U}$ of 'strict' controls constituted of $
\mathbb{F}^{\mathcal{P}}$-adapted processes each strict control $u$ taking values in the set $\mathbb{A}$, can be
considered as a G-relaxed control into the set $\mathcal{R}$ of G-relaxed
controls via the mapping
\begin{equation}
\Phi :\quad \mathcal{U}\ni u\mapsto \Phi (u)(dt,d\xi )=\delta _{\left(
u(t),u(t-\tau )\right) }\left( d\xi \right).dt\in \mathcal{R},  \label{3.4}
\end{equation}
where $\delta _{\left( u(t),u(t-\tau )\right) }$ is a Dirac measure charging 
$u(t),$ $u(t-\tau )$ for each $t.$
\begin{remark}
We mean by \textquotedblleft the process $q(\omega ,dt,d\xi )$ the $\mathbb{F}%
^{\mathcal{P}}$- progressively measurable\textquotedblright\ that for every $%
C\in \mathcal{B}(\mathbb{A})$\ and every $t\in \lbrack 0,T]$, the mapping $[0,t]\times \Omega \rightarrow \lbrack 0,1]$ described by $(s,\omega )\mapsto
\mu _{s}(\omega ,C)$ is $\mathcal{B}([0,t])\otimes \widehat{\mathcal{F}}_{t}^{\mathcal{
P}}$-measurable, and the process $(\mu _{t}(C))_{t\in \lbrack
0,T]}$ is $\mathbb{F}^{\mathcal{P}}$-adapted.
\end{remark}

The class of G-relaxed stochastic controls is denoted by $\mathcal{R}$.

To consider the control problem (\ref{3.1}), we must first study the existence and uniqueness of the solution to the following equation

\begin{eqnarray}
X\left( t\right) &=&\eta \left( 0 \right)+\int_{0}^{t}\int_{\mathbb{A}}b\left( s,X\left(
s\right) ,X\left( s-\tau \right) ,\xi _{1},\xi _{2}\right) \mu _{s}(d\xi )ds \nonumber\\ 
&+&\int_{0}^{t}\int_{\mathbb{A}}\gamma \left( s,X\left(
s\right) ,X\left( s-\tau \right) ,\xi _{1},\xi _{2}\right) \mu _{s}(d\xi)d\left\langle B\right\rangle _{s} \nonumber\\ 
&+&\int_{0}^{t}\sigma \left( s,X\left( s\right) ,X\left(
s-\tau \right) \right) dB_{s},\quad t\in \left[ 0,T\right], \label{3.5}
\end{eqnarray}

where $\mu _{t}(d\xi )=\delta _{\left( u(t),u(t-\tau )\right) }(d\xi ).$
The following assumptions are required to ensure the existence and
uniqueness of the solution of the equation $\left( \ref{3.1}\right).$

\begin{description}
\item[$(A_{1})$] The functions $b,\gamma $ and  $\sigma $ are bounded and Lipschitz continuous with respect to the space variables $x,y$ uniformly in $(t,u,u_{\tau })$.\\
There exist $K_{1}, K_{2}>0$ such that 
\begin{multline*}
\left\vert b\left( t,x,y,u,u_{\tau }\right) -b\left( t,x^{\prime },y^{\prime
},u,u_{\tau }\right) \right\vert ^{2}+\left\vert \gamma \left(
t,x,y,u,u_{\tau }\right) -\gamma \left( t,x^{\prime },y^{\prime },u,u_{\tau
}\right) \right\vert ^{2} \\
+\left\vert \sigma \left( t,x,y\right) -\sigma \left( t,x^{\prime
},y^{\prime }\right) \right\vert ^{2}\leq K_{1}\left( \left\vert x-x^{\prime
}\right\vert ^{2}+\left\vert y-y^{\prime }\right\vert ^{2}\right),
\end{multline*}
and
\begin{multline*}
\left\vert b\left( t,x,y,u,u_{\tau }\right)\right\vert ^{2}+\left\vert \gamma \left(
t,x,y,u,u_{\tau }\right)\right\vert ^{2}
+\left\vert \sigma \left( t,x,y\right) \right\vert ^{2}\leq K_{2}\left(1 + \left\vert x\right\vert ^{2}+\left\vert y\right\vert ^{2}\right),
\end{multline*}
\end{description}
for each $x,y,x^{\prime },y^{\prime
}\in \mathbb{R}^{n}$ and $u,u_{\tau }\in \mathbb{A}$.

 By the result of  \cite{Fei2019}, under our assumptions $(A_{1})$ and $(A_{2})$ the G-SDDE\thinspace (\ref{3.1}) has
a unique solution $(X_{t}^{\mu})_{t\geq 0}$, for each fixed control.
\section{Approximation and existence of G-relaxed optimal control}
\label{sec:4}

We consider a relaxed control problem\thinspace (\ref{3.3}). Let $
X^{\mu }$ denotes the solution of equation\thinspace ( \ref{3.3})
related to the G-relaxed control. Let establish the existence of a
minimizer of the cost for finite horizon $\left( T<\infty \right) $
corresponding to $\mu $.
\begin{center}
\begin{equation*}
J(\mu )=\widehat{E}\left[ \int_{0}^{T}\int_{\mathbb{A}}\mathcal{L}(t,X^{\mu
}\left( t\right) ,X^{\mu }\left( t-\tau \right) ,\xi _{1},\xi _{2})\mu
_{t}(d\xi )dt+\Psi (X^{\mu }\left( T\right) )\right],
\end{equation*}
\end{center}
where the functions,
\begin{eqnarray*}
\mathcal{L} &:&\left[ 0,T\right] \times \mathbb{R}^{n}\times \mathbb{R}%
^{n}\times \mathbb{A}\times \mathbb{A}\rightarrow \mathbb{R}, \\
\Psi &:&\mathbb{R}^{n}\mathcal{\longrightarrow }\mathbb{R},
\end{eqnarray*}%
fulfill the following important assumption
\begin{description}
\item[$\left( A_{2}\right) $] The functions $\mathcal{L}$, $\Psi $ are bounded, and the coefficient $\Psi$ is Lipschitz continuous, and $\mathcal{L}$ is Lipschitz continuous with respect to the space variables $x,y$ uniformly in time and control $(t,u,u_{\tau })$.
\end{description}
We recall that in the strict control problem
\begin{equation}
J(u)=\widehat{E}\left[ \int_{0}^{T}\mathcal{L}(t,X^{u}\left( t\right)
,X^{u}\left( t-\tau \right) ,u\left( t\right) ,u\left( t-\tau \right)
)\,dt +\Psi (X^{u}\left( T\right) )\right].  \label{3.16}
\end{equation}
From the set $\mathcal{U}$,
\begin{eqnarray}
X^{u}\left( t\right) &=&\eta(0)+\int_{0}^{t}b(s,X^{u}\left( s\right)
,X^{u}\left( s-\tau \right) ,u\left( s\right) ,u\left( s-\tau \right) )ds 
\nonumber \\
&&+\int_{0}^{t}\gamma (s,X^{u}\left( s\right) ,X^{u}\left( s-\tau \right)
,u\left( s\right) ,u\left( s-\tau \right) )d\left\langle B\right\rangle
_{s}  \nonumber \\
&&+\int_{0}^{t}\sigma \left( s,X^{u}\left( s\right) ,X^{u}\left( s-\tau
\right) \right)dB_{s},  \label{3.17}
\end{eqnarray}%
then, we have
\begin{eqnarray}
X^{\mu }\left( t\right) &=&\eta(0)+\int_{0}^{t}\int_{\mathbb{A}}b(s,X^{\mu
}\left( s\right) ,X^{\mu }\left( s-\tau \right) ,\xi _{1},\xi _{2})\mu
_{s}(d\xi )ds \nonumber \\
&&+\int_{0}^{t}\int_{\mathbb{A}}\gamma (s,X^{\mu }\left( s\right) ,X^{\mu
}\left( s-\tau \right) ,\xi _{1},\xi _{2})\mu _{s}(d\xi )d\left\langle
B\right\rangle _{s}   \nonumber\\
&&+\int_{0}^{t}\sigma \left( s,X^{\mu }\left( s\right) ,X^{\mu }\left(
s-\tau \right) \right) dB_{s}.  \label{3.18}
\end{eqnarray}
Moreover, for each $\mu $ $\in $ $\mathcal{R}$,

\begin{eqnarray}
\chi ^{\mu } &:&=\int_{0}^{T}\int_{\mathbb{A}}\mathcal{L}(t,X^{\mu }\left(
t\right) ,X^{\mu }\left( t-\tau \right) ,\xi _{1},\xi _{2})\mu _{t}(d\xi )dt
\nonumber \\
&&+\Psi (X^{\mu }\left( T\right) )\in \mathbb{L}_{G}^{1}\left( \Omega
_{T}\right) .  \label{3.19-0}
\end{eqnarray}

We use the relaxed control problem to introduce the definition of stable
convergence to define the next lemma, which according to the classical
Chattering lemma states that each G-relaxed control in $\mathcal{R}$ can
be approximated by a sequence of strict controls from $\mathcal{U}$.
\begin{definition}[\textbf{stable convergence}]
\label{stbcvg} (\cite{Elgroud2022}) Let $\mu ^{n},\mu \in \mathcal{R},n\in 
\mathbb{N}^{\ast }$. We say that, we have a stable convergence, if for any
continuous function $f:\left[ 0,T\right] \times \mathbb{A}\rightarrow 
\mathbb{R}^{n},$ we have%
\begin{equation}
\underset{n\rightarrow \infty }{\lim }\int_{\left[ 0,T\right] \times \mathbb{%
A}}f\left( t,\xi _{1},\xi _{2}\right) \mu ^{n}\left( dt,d\xi \right) =\int_{
\left[ 0,T\right] \times \mathbb{A}}f\left( t,\xi _{1},\xi _{2}\right) \mu
\left( dt,d\xi \right).  \label{3.19}
\end{equation}
\end{definition}
\begin{lemma}[G-Chattering lemma]
\label{gchtrm}Let the process $(\mu _{t})_{t\geq 0}$ is an $\mathbb{F}^{
\mathcal{P}}$-\\progressively measurable with values in $\mathcal{P}(
\mathbb{A})$. Then there exists a sequences $(u^{n}(t),u^{n}\left( t-\tau
\right) )_{n\geq 0}$ of $\mathbb{F}^{\mathcal{P}}$-progressively measurable
processes with values in $\mathbb{A},$ such that
\begin{equation*}
\mu _{t}^{n}(d\xi )dt=\delta_{\left( u^{n}\left( t\right) ,u^{n}\left( t-\tau \right)\right) }\left( d\xi \right)dt\rightarrow\mu_{t}(d\xi )dt,
\end{equation*}
converges in terms of stable convergence (thus weakly).
\end{lemma}
\begin{proof}
Given the G-relaxed control $\mu$ which is $\mathbb{F}^{\mathcal{P}}$-progressively measurable, the precise pathwise development of the approximating sequence \\$(\delta_{\left( u^{n}\left( t\right),u^{n}\left( t-\tau \right)\right) }\left( d\xi \right)dt)_{n\geq0}$ of G-relaxed control $\mu_{t}(d\xi )dt$ in $\mathcal{R}$ (see lemma after theorem 3, \cite{fleming1984}), which easily need extend to consider the strict controls $(u_{n})_{n}$ $\mathbb{F}^{\mathcal{P}}$-progressively measurable.
\end{proof}

Let $X^{n}$ the corresponding solution of G-SDDE (\ref{3.1}), associated with $u^{n}$
$\left( \text{or
}\delta _{\left( u^{n}(t),u^{n}(t-\tau )\right) }(d\xi )\right),$ and
satisfy that
\begin{eqnarray}
X^{n}\left( t\right) &=&\eta(0)+\int_{0}^{t}b(s,X^{n}\left( s\right)
,X^{n}\left( s-\tau \right) ,u^{n}\left( s\right) ,u^{n}\left( s-\tau
\right) )ds  \nonumber \\
&&+\int_{0}^{t}\gamma (s,X^{n}\left( s\right) ,X^{n}\left( s-\tau \right)
,u^{n}\left( s\right) ,u^{n}\left( s-\tau \right) )d\left\langle
B\right\rangle _{s}  \nonumber \\
&&+\int_{0}^{t}\sigma \left( s,X^{n}\left( s\right) ,X^{n}\left( s-\tau
\right) \right) dB_{s}.  \label{3.20}
\end{eqnarray}%
\bigskip The following important lemma prove the stability results for the G-SDDE (\ref{3.18}), and gives $J(\mu )=\underset{\mathbb{P}%
\in \mathcal{P}}{\sup }J^{\mathbb{P}}(\mu )$ for every$\ \mathbb{P}\in 
\mathcal{P}$.
\begin{lemma}[stability results]
\label{lsr37} Suppose that $b,\gamma $ and $\sigma $ satisfy assumption $%
(A_{1})$. Let $\mu $ be a G-relaxed control, and let $\left( u^{n}\right) 
$ be a sequence defined in Lemma \textbf{\ref{gchtrm}}. Then we have
\begin{description}
\item[$\left( i\right) $] For every$\ \mathbb{P}\in \mathcal{P},$ \textit{it
holds that} 
\begin{equation}
\underset{n\rightarrow \infty }{\lim }E^{\mathbb{P}}\left[ \underset{0\leq
t\leq T}{\sup }\left\vert X^{n}(t)-X^{\mu }(t)\right\vert^{2}\right]=0,
\label{3.23}
\end{equation}
and
\begin{equation}
\underset{n\rightarrow \infty }{\lim }\widehat{E}\left[ \underset{0\leq
t\leq T}{\sup }\left\vert X^{n}(t)-X^{\mu }(t)\right\vert^{2}\right]=0.
\label{3.24}
\end{equation}
\item[$\left( ii\right) $] Let $J(u^{n})$ and $J(\mu )$ are the corresponding cost functionals to $u^{n}$ and $\mu $ respectively. Then,
there exists a sub-sequence $\left( u^{n_{k}}\right) $ of $\left(
u^{n}\right) $ such that 
\begin{equation}
\underset{k\rightarrow \infty }{\lim }J(u^{n_{k}})=J(\mu ),  \label{3.25}
\end{equation}
and, for every$\ \mathbb{P}\in \mathcal{P}$
\begin{equation}
\ \underset{k\rightarrow \infty }{\lim }J^{\mathbb{P}}(u^{n_{k}})=J^{\mathbb{
P}}(\mu ).  \label{3.26}
\end{equation}%
\textit{Furthermore},
\begin{equation}
\inf_{u\in \mathcal{U}}J^{\mathbb{P}}(u)=\inf_{\mu \in \mathcal{R}}J^{
\mathbb{P}}(\mu ),  \label{3.27}
\end{equation}
\textit{then, there exists a }G\textit{-relaxed optimal control } $\widehat{\mu }_{
\mathbb{P}}\in \mathcal{R}$ such that%
\begin{equation}
J^{\mathbb{P}}(\widehat{\mu }_{\mathbb{P}})=\inf_{\mu \in \mathcal{R}}J^{%
\mathbb{P}}(\mu ).  \label{3.28}
\end{equation}
Since, by using the property of aggregation in Lemma $\ref{aggreg}$, under
the singularity for every $\mathbb{P}\in \mathcal{P}$, the G-SDDE $(\ref%
{3.18})$ becomes a standard SDDE, this result was confirmed by the technique
in \cite{bahlali2006}.
\end{description}
\end{lemma}
\begin{proof}
\begin{description}
\item[$\left( i\right) $] We set 
\begin{equation*}
\zeta _{n}:=\underset{0\leq t\leq T}{\sup }|X^{n}(t)-X^{{\mu }}(t)|^{2},
\end{equation*}
and note that $\zeta _{n}$ $\in \mathbb{L}_{G}^{1}(\Omega _{T})$ for each $%
n\geq 1$.
 If there is indeed a $\eta $ $>0$ such that 
\begin{equation*}
\widehat{E}[\zeta _{n}]\geq \eta, \quad n\in \mathbb{N}^{*},
\end{equation*}
we can find a probability $\mathbb{P}_{n}\in \mathcal{P}$ such that 
\begin{equation*}
E^{\mathbb{P}_{n}}[\zeta _{n}]\geq \eta -\frac{1}{n},\quad n\in \mathbb{N}^{*}.
\end{equation*}%
There exists a sub-sequence $\{\mathbb{P}_{n_{k}}\}_{k=1}^{\infty }$ that
converges weakly to some $\mathbb{P}\in \mathcal{P}$, according to $\mathcal{%
P}$ is weakly compact. Then, we have%
\begin{equation}
\underset{j\rightarrow \infty }{\lim }E^{\mathbb{P}}[\zeta _{n_{j}}]=%
\underset{j\rightarrow \infty }{\lim }\underset{k\rightarrow \infty }{\lim }%
E^{\mathbb{P}_{n_{k}}}[\zeta _{n_{j}}]\geq \underset{k\rightarrow \infty }{%
\lim \inf }E^{\mathbb{P}_{n_{k}}}[\zeta _{n_{k}}]\geq \eta.  \label{3.28.1}
\end{equation}
This is a contradiction to the fact $\underset{j\rightarrow \infty }{\lim }
E^{\mathbb{P}}[\zeta _{n_{j}}]$ $=0$ from (\ref{3.23}).\\
The proof of (\ref{3.23}) is based on G-BDG inequalities  and the standard Gronwall inequality as well as the
Dominated Convergence theorem, according to stable convergence in Lemma $\ref
{stbcvg}$ of $\delta _{\left( u^{n}\left( t\right) ,u^{n}\left( t-\tau
\right) \right) }\left( d\xi \right) dt$ converges to $\mu _{t}\left( d\xi \right) dt$, and the proof method does not extend to proving (\ref{3.24}) because the
Dominated Convergence theorem (and even the celebrated Fatou's lemma) is no
longer valid under sublinear expectation, but the G-BDG inequalities hold
true for G-stochastic integrals and G-SDDEs.
\end{description}
\begin{description}
\item[$\left( ii\right) $] Assume that $\mathcal{L}$ and $\Psi $ satisfy
assumption $(A_{2})$. \\
We have $dt\delta _{\left( u^{n}\left( t\right)
,u^{n}\left( t-\tau \right) \right) }\left( d\xi \right) $ converges weakly
to $dt\mu _{t}\left( d\xi \right) $ quasi-surely. Then, there exists a
sub-sequence $(u^{n_{k}})$ of $(u^{n}),$ we obtain
\begin{equation}
\underset{k\rightarrow \infty }{\lim }J(u^{n_{k}})=J(\mu ).  \label{111}
\end{equation}
 By using Proposition 17 in \cite{denis2011} and (\ref{3.24}) it follows that, there exists a sub-sequence $(X^{n_{k}}(t))_{n_{k}}$
that converges quasi-surely to $X^{{\mu}}(t)$ $i.e.$ $\mathbb{P}$-$a.s.$, for all $\mathbb{P}\in \mathcal{P}$,
uniformly in $t$. We can use (\ref{3.24}) for every $\mathbb{P%
}\in \mathcal{P},$ to obtain that%
\begin{equation*}
\ \underset{k\rightarrow \infty }{\lim }J^{\mathbb{P}}(u^{n_{k}})=J^{\mathbb{%
P}}(\mu ).
\end{equation*}
 From the notation (\ref{3.19-0}), we can note that 
\begin{equation*}
J(u^{n_{k}})=\widehat{E}[\chi ^{u^{n_{k}}}],
\end{equation*}
and
\begin{equation*}
J({\mu})=\widehat{E}[\chi ^{\mu }],
\end{equation*}
 where both $\chi ^{u^{n_{k}}}$, $\chi ^{\mu }$ $\in $ $\mathbb{L}%
_{G}^{1}(\Omega _{T})$.
 If there is some $\eta $ $>0$ then $\ \ $%
\begin{equation*}
\widehat{E}[\chi ^{u^{n_{k}}}]\geq \widehat{E}[\chi ^{\mu }]+\eta, \quad n_{k}\geq
l, l+1,...,
\end{equation*}
 thus, we can find a probability measure $\mathbb{P}_{m}\in \mathcal{P} 
$ such that%
\begin{equation*}
E^{\mathbb{P}_{m}}[\chi ^{u^{n_{k}}}]\geq \widehat{E}[\chi ^{\mu }]+\eta -%
\frac{1}{m}.
\end{equation*}
 Then we can find a sub-sequence $\{\mathbb{P}_{m_{k}}\}_{k=1}^{\infty }$
under a weakly compact $\mathcal{P}$, that converges to $\mathbb{P}\in 
\mathcal{P}$. We have%
\begin{eqnarray*}
E^{\mathbb{P}}[\chi ^{\mu }] &=&\underset{k\rightarrow \infty }{\lim }E^{%
\mathbb{P}_{m_{k}}}[\chi ^{\mu }]=\underset{k\rightarrow \infty }{\lim }%
\underset{j\rightarrow \infty }{\lim }E^{\mathbb{P}_{m_{k}}}[\chi
^{u^{n_{j}}}]\geq \underset{j\rightarrow \infty }{\lim \inf }E^{\mathbb{P}%
_{m_{j}}}[\chi ^{u^{n_{j}}}] \\
&\geq &\underset{j\rightarrow \infty }{\lim \inf }\left( \widehat{E}[\chi
^{\mu }]+\eta -\frac{1}{m_{j}}\right) \\
&=&\widehat{E}[\chi ^{\mu }]+\eta .
\end{eqnarray*}
 Thus, 
\begin{equation*}
E^{\mathbb{P}}[\chi ^{\mu }]\geq \widehat{E}[\chi ^{\mu }]+\eta,
\end{equation*}
 from the definition of the sublinear expectation, we obtain a contradiction.
 Therefore,
\begin{equation*}
\underset{k\rightarrow \infty }{\lim }J(u^{n_{k}})\leq J(\mu ).
\end{equation*}
By using ( \ref{3.26}) to prove that%
\begin{equation*}
\underset{k\rightarrow \infty }{\lim }J(u^{n_{k}})\geq J(\mu ).
\end{equation*}
We have 
\begin{eqnarray*}
\underset{k\rightarrow \infty }{\lim }J(u^{n_{k}}) &\geq &\underset{
k\rightarrow \infty }{\lim }J^{\mathbb{P}}(u^{n_{k}}),\quad\mathbb{P}\in 
\mathcal{P}\\
&=&J^{\mathbb{P}}(\mu ),\quad\mathbb{P}\in \mathcal{P}.
\end{eqnarray*}
\end{description}
Therefore, $\underset{k\rightarrow \infty }{\lim }J(u^{n_{k}})\geq $ $J(\mu )$.
\end{proof}

The following theorem constitutes the main result of our problem, which gives that the two problems have the same Infinium of the expected costs.
\begin{theorem}
\label{relaxed} For every $u\in \mathcal{U}$ and $\mu \in \mathcal{R}$, we
have 
\begin{equation}
\inf_{u\in \mathcal{U}}J(u)=\inf_{\mu \in \mathcal{R}}J(\mu ).  \label{3.43}
\end{equation}
Furthermore, there is a G-relaxed optimal control $\widehat{\mu }\in 
\mathcal{R}$ such that%
\begin{equation}
J(\widehat{\mu })=\min_{\mu \in \mathcal{R}}J(\mu ),  \label{3.44}
\end{equation}
recall that
\begin{equation}
J(\mu )=\sup_{\mathbb{P}\in \mathcal{P}}J^{\mathbb{P}}(\mu ),  \label{3.45}
\end{equation}
where for each $\mathbb{P}\in \mathcal{P}$, the relaxed cost functional is
given as follow%
\begin{equation}
J^{\mathbb{P}}(\mu )=E^{\mathbb{P}}\left[ \int_{0}^{T}\int_{\mathbb{A}}%
\mathcal{L}(t,X^{\mu }\left( t\right) ,X^{\mu }\left( t-\tau \right) ,\xi
_{1},\xi _{2})\mu _{t}(d\xi )dt +\Psi (X^{\mu }\left( T\right) )\right]. \label{3.46}
\end{equation}
\end{theorem}
To prove (\ref{3.43}), using Lemma \thinspace \ref{gchtrm} of G-Chattering lemma  and Lemma \thinspace \ref{lsr37} of stability results for the G-SDDE (\ref{3.18}). According to Lemma \thinspace \ref{gchtrm},
given a G-relaxed control $\mu \in \mathcal{R}$, and there is a sequence $
(u^{n})_{n}\in $ $\mathcal{U}$ of strict controls such that $\delta _{\left(
u^{n}(t), u^{n}\left( t-\tau \right) \right) }(d\xi )dt$ converges weakly to $%
\mu _{t}(d\xi )dt$ quasi-surely i.e. $\mathbb{P}$-$a.s.$, for all $\mathbb{P}
$ $\in $ $\mathcal{P}$. The existence of a G-relaxed optimal control for
each $\mathbb{P}$ $\in $ $\mathcal{P}$ and a tightness argument is used to
prove (\ref{3.44}).\\
\begin{proof}
From  (\ref{3.4}) and (\ref{111}) we can easily obtain that
\begin{equation*}
\inf_{u\in \mathcal{U}}J(u)\leq \inf_{\mu \in \mathcal{R}}J(\mu ).
\end{equation*}
Therefore, for every $u\in \mathcal{U}$, $\delta _{u}\in \mathcal{R}$, we
have
\begin{equation*}
J(u)=J(\delta _{u})\geq \inf_{\mu \in \mathcal{R}}J(\mu ).
\end{equation*}
Hence,
\begin{equation*}
\inf_{u\in \mathcal{U}}J(u)\geq \inf_{\mu \in \mathcal{R}}J(\mu ),
\end{equation*}
which proves (\ref{3.43}).
Since $\mathcal{L}$ and $\Psi $ are continuous and bounded, for each $%
\upsilon \in \mathcal{R},$ we turn now to the proof of existence of G-relaxed
optimal control.
\begin{equation*}
\chi ^{\mu }:=\int_{0}^{T}\int_{\mathbb{A}}\int_{\mathbb{A}}\mathcal{L}%
(t,X^{\upsilon }\left( t\right) ,X^{\upsilon }\left( t-\tau \right) ,\xi
_{1},\xi _{2})\mu _{t}(d\xi )dt+\Psi (X^{\upsilon }\left( T\right) )\in 
\mathbb{L}_{G}^{1}\left( \Omega _{T}\right).
\end{equation*}
By using Lemma \thinspace \ref{gh363}, we can deduce that for each $
\upsilon \in \mathcal{R}$,
\begin{equation}
\underset{n\rightarrow \infty }{\lim }J^{\mathbb{P}_{n}}(\upsilon )=J^{%
\mathbb{S}}(\upsilon ),  \label{3.47}
\end{equation}
Then, the sequence $\{\mathbb{P}_{n}\}_{n=1}^{\infty }\in \mathcal{P}$
converges weakly to $\mathbb{S}$ $\in \mathcal{P}$ .
Assume there is an $\varsigma >0$ such that, for each $\upsilon \in \mathcal{%
R},$
\begin{equation*}
J(\upsilon )\geq \inf_{\mu \in \mathcal{R}}J(\mu )+\varsigma.
\end{equation*}
Next, by according to Lemma \thinspace \ref{lsr37}, for each $\mathbb{P}\in 
\mathcal{P}$ there exists a G-relaxed optimal control $\widehat{\mu }\in \mathcal{
R}$ such that
\begin{equation*}
\widehat{\mu }_{\mathbb{P}}=\arg \left( \underset{\mu \in \mathcal{R}}{\min }
J^{\mathbb{P}}(\mu )\right),
\end{equation*}
we get that
\begin{eqnarray*}
J(\upsilon ) &\geq &\underset{\mathbb{P}\in \mathcal{P}}{\sup }\left(
\inf_{\mu \in \mathcal{R}}J^{\mathbb{P}}(\mu )\right) +\varsigma \\
&=&\underset{\mathbb{P}\in \mathcal{P}}{\sup }J^{\mathbb{P}}(\widehat{\mu }_{%
\mathbb{P}})+\varsigma .
\end{eqnarray*}
However, there exists $\mathbb{P}_{n}\in \mathcal{P}$ for every $n\geq 1$, such that
\begin{equation*}
J^{\mathbb{P}_{n}}(\upsilon )\geq J(\upsilon )+\frac{1}{n}.
\end{equation*}
We can extract a sub-sequence $\{\mathbb{P}_{n_{j}}\}_{j=1}^{\infty }\in 
\mathcal{P}$, from the sequence $\{\mathbb{P}_{n}\}_{n=1}^{\infty }\in 
\mathcal{P}$ being weakly compact, which converges weakly to some $\mathbb{S}
$ $\in \mathcal{P}$.
Thus, As a result of \thinspace (\ref{3.47}) for each $\upsilon \in 
\mathcal{R}$, it follows that
\begin{equation*}
J^{\mathbb{S}}(\upsilon )=\underset{j\rightarrow \infty }{\lim }J^{\mathbb{P}%
_{n_{j}}}(\upsilon )\geq \underset{\mathbb{P}\in \mathcal{P}}{\sup }J^{%
\mathbb{P}}(\widehat{\mu }_{\mathbb{P}})+\varsigma .
\end{equation*}
Particularly, we obtain for a given $\upsilon ^{\mathbb{S}}$ $\in $ 
$\mathcal{R}$,
\begin{equation*}
J^{\mathbb{S}}(\upsilon ^{\mathbb{S}})\geq \underset{\mathbb{P}\in \mathcal{P%
}}{\sup }J^{\mathbb{S}}(\upsilon ^{\mathbb{S}})+\varsigma,
\end{equation*}
which is a contradiction with the fact that $\varsigma >0$.
\end{proof}
\section{Economics model}
\label{sec:5}
As an application of our theoretical result, an economics model describing the rate of change of capital $K$ and labor $L$ in a market by a system of ODEs introduced by F. R. Ramsey in $1928$ (for more details, can see \cite{Ramsey1928}), given by
\begin{equation}
\left\{ 
\begin{array}{c}
\frac{dK_{s}}{ds}=P_{s}-C_{s},\\ 
\frac{dL_{s}}{ds}=a_{s}L_{s},
\end{array}
\right.  \label{FNM}
\end{equation}
where 
\begin{center}
\begin{description}
\item[$P$]: production,
\item[$C$]: consumption rates.
\item[$a_{s}$]: the rate of growth of labor (population),
\end{description}
\end{center}
In \cite{Gandolfo1997} the production, capital and labor
related by the Cobb-Douglas formula as follow
\begin{equation}
P_{s} =DK_{s}^{\alpha }L_{s}^{\beta }, \label{CobDg}
\end{equation}
where $D,\alpha,\beta > 0$ are a constants. If $\alpha=\beta =1$, we get the linearity of (\ref{CobDg}) under certain scenarios, which we will assume throughout our problem. The labor is constantly given by $L_{s}=L_{0}$, which verifies specific markets, or we take in many years relatively short time intervals. As a result, the production rate and capital have a dependent formula $P_{s}=MK_{s}$, where $M=DL_{0}$.

Large random disturbances (big noise), which come from political decisions, wars, and the atmosphere situations, are the real reasons for influencing the change in the economy. Then, we can define the production rate by
\begin{equation}
 P_{s}=MK_{s}+"\textit{big noise}".   \label{prod}
\end{equation}
 By using (\ref{prod}) in (\ref{FNM}), we obtain
\begin{equation*}
\frac{dK_{s}}{ds}=MK_{s}-C_{s}+"\textit{big noise}".
\end{equation*}
Then, we obtain the following system
\begin{equation}
dK_{s}=[MK_{s}-C_{s}]ds+\gamma \sigma (K_{s})dW_{s},
\end{equation}
where $W_{s}$ is a standard Brownian motion,
with $\gamma$ is an unknown parameter and the coefficient $\sigma$ depends on $K$ because is strongly affected by factors that change the economy due to a big noise. 
Usually, the essential period is needed for transition due to the influence and change of the economy, such as the length of the production cycle in many economic situations during the war such as the increase in oil and the impact on wheat production.
Therefore, the most accurate assumption given 
the change in the rate of capital $K$ depends on the investment made at time $s-\tau $, where $s$ is the current time and is the duration of the cycle required to create working capital during wartime delays and political decisions delays, and as time passes, the noise increases.
\begin{equation}
dK_{s}=[MK_{s-\tau}-C_{s}]ds+\gamma \sigma (K_{s-\tau})dW_{s}, \label{cap-del}
\end{equation}
which means that we can not estimate the value of $\gamma$, the only information that we can know is that gamma in $[\sigma _{min},\sigma
_{max}]$, for some $\sigma _{min},\sigma _{max}\in $ $\mathbb{R}_{+}^{\ast },$
the idea is to consider the worst-case scenario; i.e., 
\begin{equation*}
\sup_{\gamma\in [\sigma _{min},\sigma
_{max}]}E^{\mathbb{P}_{\gamma}}[\phi(K) ]=\widehat{E}[\phi(\widehat{K})],
\end{equation*}
where $E^{\mathbb{P}_{\gamma}}$ is an expected valued over $\mathbb{P}_{\gamma}$, with $\mathbb{P}_{\gamma}$ is the probability indexed induced by (\ref{cap-del}), and $\widehat{E}$ is the non-linear expectation, with $\phi$ is a given function, and $\widehat{K}$ is the solution of the following equation
\begin{equation}
d\widehat{K}_{s}=[M\widehat{K}_{s-\tau}-C_{s}]ds+\sigma (\widehat{K}_{s-\tau})dB_{s}, \label{Ncap-del}
\end{equation}
where $B_{s}$ is the G-Brownian motion follow  $\mathcal{N}(0,$ $s[\sigma^{2}_{min},\sigma^{2}_{max}])$. It is necessary to study how we make decisions when we are often faced with probabilistic uncertainty. The main question is to control the consumption rate $C$ by a control $u_{.}$, i.e., it is of the form $u_{s}C_{s}$, then
\begin{equation}
d\widehat{K}_{s}=[M\widehat{K}_{s-\tau }-u_{s}C_{s}]ds+\sigma \left( \widehat{K}_{s-\tau
}\right) dB_{s},  \label{2}
\end{equation}
with
\begin{equation}
\widehat{K}_{\theta}= k\left( \theta\right), \quad t-\tau \leq \theta\leq t.
\label{3}
\end{equation}
For this, we suggest studying the modified Ramsey model uncertainty
with delay introduced by (\ref{2}) and (\ref{3}). We want to minimize the investment capital under the above assumptions, by the following cost function
\begin{equation*}
J\left( u;t,k\right) =\widehat{E}\left[\left[ \widehat{K}_{T}-\alpha \right]^{2} +\frac{1}{2}\int_{t}^{T}
u_{r}^{2}dr \right],
\end{equation*}
The idea is to get our capital $\widehat{K}$ in the future at time $T$ closer to $\alpha$ with minimal energy, with $\alpha$ as a random variable.
\subsection{Optimization problem}

To solve our optimal control, we connect its HJB with a G-FBSDDE that we solve numerically. In fact, the condition of quasi-continuity is not required for random variables and processes in \cite{peng2010}.

As a result, see \cite{peng2010} may be applied to practical situations that are not in quasi-continuous spaces. We assume some independence among components of G-Brownian motion in the framework of super-linear expectation and derive an obvious type of HJB equation with a vector-valued control variable. It can be observed that the result in the next provides an application of stochastic control and the uncertainty volatility model.

Let define $(K_{t})_{t\in[0,T]}$, as a $C$-valued stochastic processes from the space $S^{p}([0,T ]; C)  (1 \leq p < \infty )$, and the norm defined by
\begin{equation}
  \left\Vert {K}\right\Vert^{p}_{S^{p}([0,T ]; C)}=\widehat{E}[ \underset{t\in [0, T]} {\sup} \left \vert K_{t} \right \vert^{p}_{C}]=\widehat{E}[ \underset{t\in [0, T]} {\sup} \underset{\theta\in  [-\tau, 0]} {\sup} \left \vert K_{t}(\theta) \right \vert^{p}]< \infty,
\end{equation}
with C-continuous paths.

The gradient $\nabla_{k}V(t,k)$ at $(t,k)\in[0,T]\times C\left(\left[-\tau ,0\right] ;\mathbb{R}^{n}\right)$ is an $n$-tuple of finite Borel measures on $[-\tau,0]$. Let define
\begin{equation}
\nabla_{0} V(t, k)=\nabla_{k}V(t, k)\{(0)\},
\end{equation}

with $\left \vert \nabla_{k}V(t,k)\right \vert$ is a  total
variation norm and define
where $\nabla_{0} V(t, k)$
is a vector in $\mathbb{R}^{n}$  which represent the
masse at point $0$ of the component of $\nabla_{k} V(t, k)$,  such that we have the following continuous map on $[0, T] \times C$ for every $g\in C$\\
\begin{equation}
(t,k)\mathcal{\longrightarrow }\left\langle {\nabla_{k} V(t, k), g}\right\rangle =\int_{[-\tau, 0]}g(\theta)\cdot \nabla_{k}V(t, k)(d\theta),
\end{equation}
for more details see \cite{fuhrman2010}.
Then the value function define by
\begin{equation}
V\left(t,k\right) =\underset{u}{\inf }\widehat{E}\left[ \left[
K_{T}-\alpha \right]^{2} +\frac{1}{2}\int_{t}^{T} u_{r}^{2}dr\mid
_{K_{t}^{u}=k}\right].  \label{VF}
\end{equation}
Then, the HJB equation given by

\begin{equation}
\left\{ 
\begin{array}{l}
\frac{\partial V\left(t,k\right)}{\partial{t}}
+\underset{u}{\inf }\left\{ LV\left(t,k\right)+\frac{1}{2}u^{2}\right\} =0, \\ 
V\left(T,k\right)=\left[ k-\alpha \right]^{2},
\end{array}
\right.  \label{G-HJB1}
\end{equation}

where $u$ is a control variable is selected dynamically within $\mathbb{R}^{n}$. 

The infinitesimal generator $L$ given by
\begin{equation}
L =\frac{1}{2} \sigma^{2}\left(
k\right) \nabla_{0}^{2}+Mk\nabla_{0}-uC_{t}\nabla _{0},  \label{IGEN}
\end{equation}
then, we get
\begin{eqnarray}
&\frac{\partial V\left(t,k\right)}{\partial{t}}+\frac{1}{2} \sigma ^{2}\left( k\right)\nabla _{0}^{2}V\left( t,k\right)+Mk\nabla _{0}V\left( t,k\right) \nonumber\\
&+\underset{u}{\inf } \left\{ -uC_{t}\nabla_{0}V\left(t,k\right)+\frac{1}{2}u^{2}\right\}=0.
\label{G-HJB2}
\end{eqnarray}
Hence, the infimum is achieved when
\begin{equation*}
u^{\ast }-C_{t}\nabla _{0}V\left( t,k\right) =0.
\end{equation*}
Therefore
\begin{equation*}
u^{\ast }=C_{t}\nabla _{0}V\left(t, k\right).
\end{equation*}
As a converse of Eq. (\ref{G-HJB2}), we have
\begin{equation}
\left\{ 
\begin{array}{l}
\frac{\partial V\left(t,k\right)}{\partial{t}}+\frac{1}{2} \sigma ^{2}\left( k\right)
\nabla _{0}^{2}V\left(t, k\right)+Mk\nabla _{0}V\left(t, k\right)-\frac{1}{2}C_{t}^{2}\left(
\nabla _{0}V\left(t, k\right)\right) ^{2}=0, \\ 
V\left(T,k\right)=\left[ k-\alpha \right]^{2},
\end{array}%
\right. \label{PDE}
\end{equation}
where $V\in C_{b}^{1,2}$ define the processes $Y^{t,k}_{s}=V\left( s, k\right) $ and 
\begin{eqnarray}
Z^{t,k}_{s}&=&\sigma \left(k\right) \nabla _{0}V\left( s,k\right),   \label{NOT} 
\end{eqnarray}
where $\sigma $ is non-degenerate diffusion$.$
\begin{eqnarray}
Y^{t,k}_{s} &=&\left[ \widehat{K}^{t,k}_{T}-\alpha \right]^{2} -\frac{1}{2}\int_{s}^{T}\left[ 
C_{r}Z^{t,k}_{r}\sigma ^{-1}\left( \widehat{K}^{t,k}_{r-\tau }\right) \right]^{2}  dr \nonumber\\
&+&\int_{s}^{T}Z^{t,k}_{r}dB_{r}, \quad\label{G-BSDE}
\end{eqnarray}
and the G-FBSDDEs defined on $[t,T]\subseteq [0,T]$: for $s\in  [t,T]$,
\begin{eqnarray}
\left\{ 
\begin{aligned}
\widehat{K}^{t,k}_{s}&=k(t) +\int_{t}^{s}\left( M\widehat{K}^{t,k}_{r-\tau }\right)
dr+\int_{t}^{s} \sigma \left( \widehat{K}^{t,k}_{r-\tau }\right) dB_{r}, \\ 
Y^{t,k}_{s}&=\left[ \widehat{K}^{t,k}_{T}-\alpha \right]^{2} -\frac{1}{2}\int_{s}^{T}\left[ C_{r}Z^{t,k}_{r}\sigma^{-1}\left( \widehat{K}^{t,k}_{r-\tau }\right)\right]^{2}dr\\
&+\int_{s}^{T}Z^{t,k}_{r}dB_{r}.
\end{aligned}
\right.  \label{G-FBSDEs}
\end{eqnarray}

\subsection{Numerical study}
In this section, we present numerical results for G-FBSDE with delay (\ref{G-FBSDEs}) based on a variant of the least-squares Monte Carlo algorithm proposed in \cite{bender2012}. It comprises two phases, as do most numerical methods for G-FBSDE:

\begin{itemize}
\item The explicit time discretization of G-FBSDE with delay.
\item The approach for approximation of conditional expectations.
\end{itemize}
First, for the time discretization, we consider a partition of the
interval $[-\tau ,T]$ as $\pi =\left\{ t_{0},t_{1},...,t_{N}\right\} $, $%
i.e. $, $-\tau =t_{0}<t_{1}<t_{2}<$\textperiodcentered \textperiodcentered
\textperiodcentered $<t_{N_{\tau}}=0<t_{N_{\tau}+1}<...<t_{N}=T$.
For the pair $(\widehat{K}_{s},Y_{s}).$ We motivate a natural time discretization that
goes backward in time. Defining $\Delta t_{n}=t_{n+1}-t_{n},$ $\Delta
B_{t_{n}}=B_{t_{n+1}}-B_{t_{n}}$ for $
t_{n}\in \left\{ t_{N_{\tau}},t_{N_{\tau}+1},...,t_{N}\right\}$, we have
\begin{equation}
\left\{ 
\begin{array}{l}
\widehat{K}_{t_{n}}=\widehat{K}_{t_{n+1}}+M\widehat{K}_{t_{n}-\tau}\Delta t_{n}
+\sigma \left( \widehat{K}_{t_{n}-\tau}\right) \Delta B_{t_{n}}, \\ 
Y_{t_{n}}=Y_{t_{n+1}}+\frac{1}{2}\left[ C_{t_{n}}Z_{t_{n+1}}\sigma ^{-1}\left(
\widehat{K}_{t_{n}-\tau }\right) \right]^{2} \Delta t_{n}-Z_{t_{n+1}}\Delta B_{t_{n}},
\end{array}
\right.\label{Ab}
\end{equation}
where $Y_{t_{N}}=\left[ \widehat{K}_{t_{N}}-\alpha \right]^{2}.$ 

By taking a conditional expectation on both sides of G-BSDE in (\ref{Ab}) to get an approximation of $Y_{t_{n}}$, given $Y_{t_{n+1}}$ from the joint process $(\widehat{K}_{t_{n}},Y_{t_{n}})$ which is adapted to the filtration generated by $(B_{r})_{0\leq r\leq s}$, and by using the property of symmetric martingale, we obtain that
\begin{eqnarray}
Y_{t_{n}}=\widehat{E}\left[ Y_{t_{n}}|\mathcal{F}_{t_{n}}\right] &=&\widehat{E}%
\left[( Y_{t_{n+1}}+\frac{1}{2}\left[ C_{t_{n}}Z_{t_{n}}\sigma ^{-1}\left(
\widehat{K}_{t_{n}-\tau }\right) \right]^{2} \Delta t -Z_{t_{n}}\Delta B_{t_{n}})|\mathcal{F}_{t_{n}}\right] \nonumber\\
&=&\widehat{E}\left[( Y_{t_{n+1}}+\frac{1}{2}\left[ C_{t_{n}}Z_{t_{n}}\sigma ^{-1}\left(
\widehat{K}_{t_{n}-\tau }\right) \right]^{2} \Delta t |\mathcal{F}_{t_{n}}\right] \nonumber\\
&+&\widehat{E}\left[ -Z_{t_{n}}\Delta B_{t_{n}}|\mathcal{F}_{t_{n}}\right] \nonumber\\
 &=&\widehat{E}
\left[( Y_{t_{n+1}}+\frac{1}{2}\left[ C_{t_{n}}Z_{t_{n}}\sigma ^{-1}\left(
\widehat{K}_{t_{n}-\tau }\right) \right]^{2} \Delta t)|\mathcal{F}_{t_{n}}\right],
\label{Ac1}
\end{eqnarray}
from (\ref{Ac1}) we have the time discretization $(Y_{t_{n}}^{\pi },Z_{t_{n}}^{\pi })$ for $(Y,Z)$ by
\begin{equation*}
Y_{t_{n}}^{\pi }=\widehat{E}\left[ Y_{t_{n}}^{\pi }|\mathcal{F}_{t_{n}}
\right] =\widehat{E}\left[( Y_{t_{n+1}}^{\pi }+\frac{1}{2}\left[ C_{t_{n}}^{\pi }Z_{t_{n}}^{\pi }\sigma ^{-1}\left(
\widehat{K}_{t_{n}-\tau }^{\pi}\right) \right]^{2}
\Delta t)|\mathcal{F}_{t_{n}}\right].
\end{equation*}

Let use the value of $Z_{s}= \sigma \left(
k\right) \nabla _{0}V(s,k)$ $\left( \text{resp. }
Z_{t_{n+1}}= \sigma \left( k_{t_{n}}\right)
\nabla _{0}V(t_{n},k_{t_{n}})\right) $ to get $Y_{t_{n}}$. 
By least square method \cite{kebiri2019},
\begin{equation}
Y_{t_{n}}=\underset{Y\in L^{2}}{\arg \min }\widehat{E}\left[ \left\vert
Y\left( \widehat{K}_{t_{n}}\right) -Y_{t_{n+1}}^{\pi }-\frac{1}{2}\left[ C_{t_{n}}^{\pi }Z_{t_{n+1}}^{\pi }\sigma
^{-1}\left( \widehat{K}_{t_{n}-\tau }^{\pi }\right) \right] ^{2}\Delta
t\right\vert ^{2}\right],  \label{abcd}
\end{equation}
$Y$ iterates through all measurable functions with $\widehat{E}\left[
\left\vert Y\left( \widehat{K}_{t_{n}}\right) \right\vert ^{2}\right] <$ $\infty$.
Finally, we get that 
\begin{equation}
Y_{t_{n}}=\underset{Y=Y\left( \widehat{K}_{t_{n}}\right) }{\arg \min }%
\widehat{E}\left[ \left\vert Y\left( \widehat{K}_{t_{n}}\right)
-Y_{t_{n+1}}^{\pi }-\frac{1}{2}\left[ C_{t_{n}}^{\pi }Z_{t_{n+1}}^{\pi }\sigma
^{-1}\left( \widehat{K}_{t_{n}-\tau }^{\pi }\right) \right] ^{2}\Delta
t\right\vert ^{2}\right] .  \label{abcde}
\end{equation}

By using $Y\left( \widehat{K}_{t_{n}}\right) \approx \underset{i=1}{\overset{L_{1}}{\sum }}%
\varpi _{i}\left( t_{n}\right) \chi _{i}\left( \widehat{K}_{t_{n}}\right),$ for some $L_{1}\in \mathbb{N}$, where $\varpi _{1}\left( .\right)\\
,...,\varpi _{L_{1}}\left( .\right) \in \mathbb{R}
$ is the coefficients of the basis $\chi _{1},...,\chi _{L_{1}}:\mathbb{R}^{n}\rightarrow 
\mathbb{R}$. By least-square method

\begin{equation}
\varpi ^{\left( L_{1},L_{2}\right) }\left( t_{n}\right) =\underset{\alpha
\in \mathbb{R}^{L_{1}}}{\arg \min }\left\Vert W_{t_{n}}^{\left(
L_{1},L_{2}\right) }\varpi -\mathcal{A}_{t_{n}}\right\Vert ^{2},  \label{LSM}
\end{equation}
where $W_{t_{n}}^{\left( L_{1},L_{2}\right) }=\left( \chi _{i}\left(
\widehat{K}_{t_{n}}^{\pi }\right) \right) _{i=1,...,L_{2},i=1,...,L1}.$
then, we get that
\begin{equation*}
\varpi ^{\left( L_{1},L_{2}\right) }\left( t_{n}\right) =\left( \left(
W_{t_{n}}^{\left( L_{1},L_{2}\right) }\right) ^{T }\left(
W_{t_{n}}^{\left( L_{1},L_{2}\right) }\right) \right) ^{-1}\left(
W_{t_{n}}^{\left( L_{1},L_{2}\right) }\right)^{T } \mathcal{A}_{t_{n}},
\end{equation*}
is the solution to the least-squares problem (\ref{LSM}), and for further information (see \text{\cite{kebiri2018}}).\\
In order to solve the forward G-SDDE, we have first to simulate the increment of the G-Brownian motion, for this later we use the same approach provided on \cite{yang2016} to simulate the density (resp. distribution) of the G-Normal BM  by employing the finite difference method to simulate its corresponding G-PDE.

The simulated result in Figure \thinspace \ref{G-Density} (resp. Figure \thinspace \ref{G-Normal}) illustrate the G-Normal density (resp. distribution) with $\sigma _{min}=0.75$ and $\sigma _{\max }=1.25$.
\begin{figure}[h!]
\centering
\includegraphics[width=0.55\textwidth]{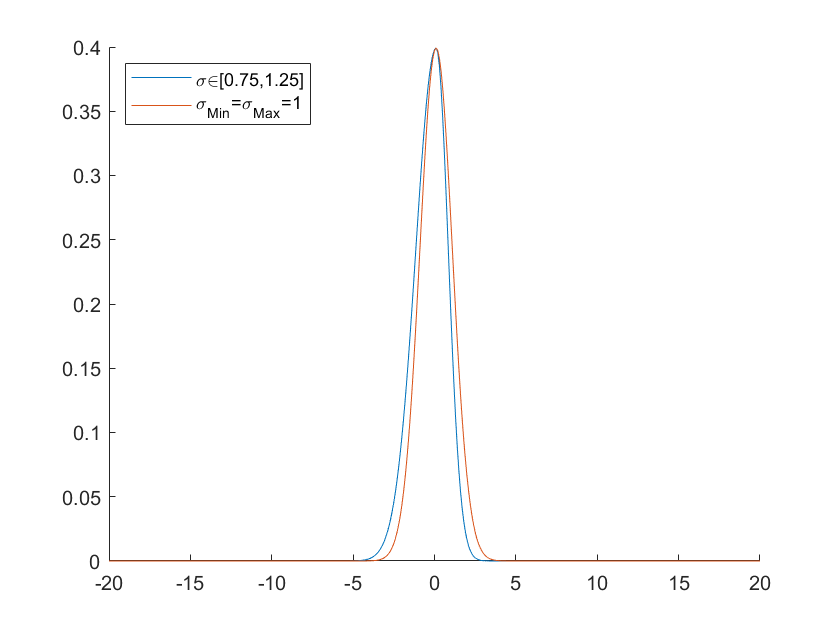}
\caption{The G-Normal density for $\sigma _{min}=0.75, 
\sigma_{max}=1.25$ and the standard normal density.}\label{G-Density}
\end{figure}
\begin{figure}[hb!]
\centering
\includegraphics[width=0.53\textwidth]{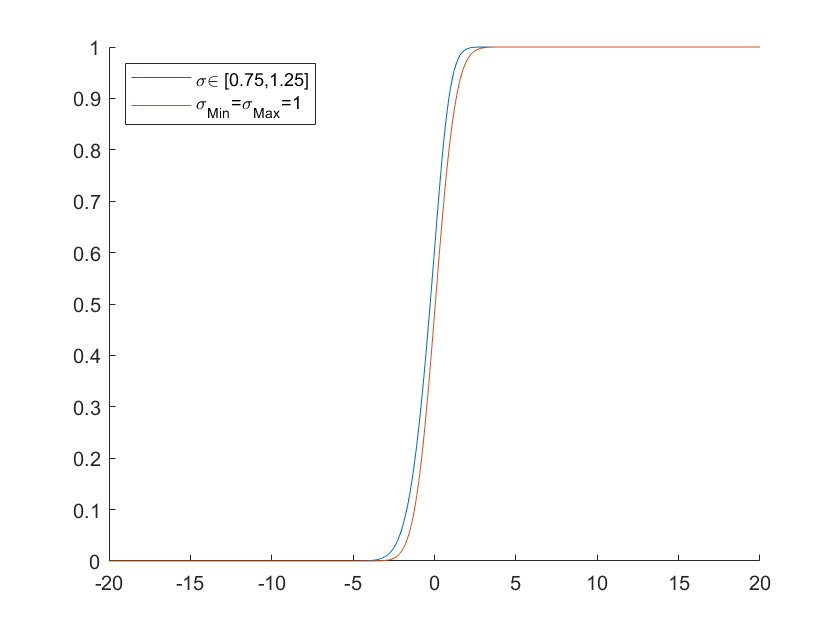}
\caption{The G-Normal-distribution for $\sigma _{min}=0.75,\,
\sigma_{max}=1.25$ and the standard normal density.}\label{G-Normal}
\end{figure}
\clearpage
 Concerning the choice of the parameters, we take the value of $ \sigma_{min}=0.75, \sigma_{max}=1.25, T=1, M=1, \tau =0.1, \alpha =0,$ and $\sigma (\widehat{K})=2*\widehat{K}$. We take the value of $\widehat{K}$ in the interval $[-\tau,0]$ are uniform random values in the interval $[1,2].$ For this choice of coefficients and data, we obtain the following simulation result: 
Figure \thinspace \ref{G-FSDEwD}, shows that the trajectories of the solution of the G-FSDDE corresponding to $\sigma _{max}=1.25, \sigma _{min}=0.75.$
\begin{figure}[ht!]
\centering
\includegraphics[width=0.7\textwidth]{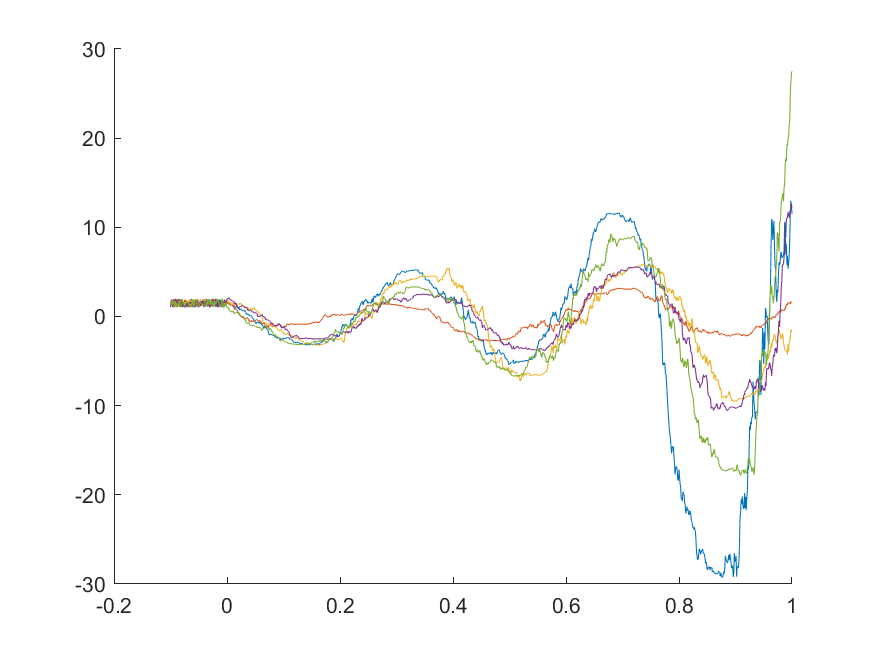}
\caption{The solution of the G-FSDDE.}\label{G-FSDEwD}
\end{figure}

In Figure \thinspace \ref{G-BSDEwD}, we represent the trajectories of the solution of the G-BSDDE.  According to a least-square Monte Carlo scheme and based on the Euler discretization.

\begin{figure}[hb!]
\centering
\includegraphics[width=0.68\textwidth]{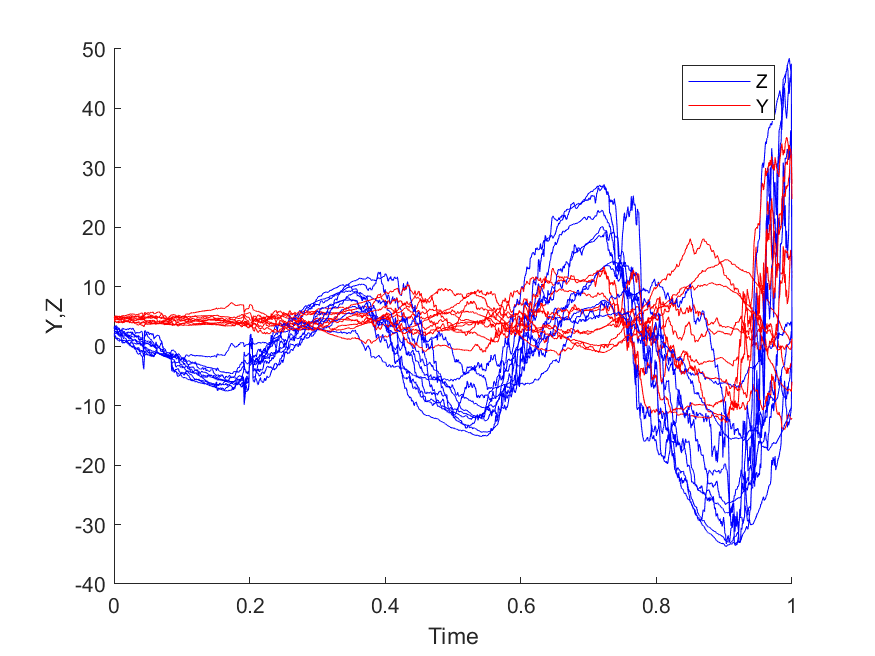}
\caption{The solution of the G-BSDDE.} \label{G-BSDEwD}
\end{figure}

The simulation of G-FBSDDE gives us the value of Y at initial time $Y_{0}=3.9281$, which represents the optimal cost of our optimal control.
\section*{Acknowledgements}
The author's knowledge of the funding of the ERASMUS KA107 project
\\- Omar Kebiri knowledge of the funding of the Deutsche Forschungsgemeinschaft (DFG, German Research Foundation) under \ Germany's Excellence Strategy. The Berlin Mathematics Research Center MATH+ (EXC-2046/1, project ID: 390685689), project EF4-6.

\bibliographystyle{plain}
\bibliography{bibdel}

\begin{thebibliography}{10}

\bibitem{Agram2014}
Nacira Agram and Bernt {\O}ksendal.
\newblock Infinite horizon optimal control of forward-backward stochastic
  differential equations with delay.
\newblock {\em Journal of Computational and Applied Mathematics}, 259:336--349,
  2014.

\bibitem{bahlali2006}
Se{\"\i}d Bahlali, Brahim Mezerdi, and Boualem Djehiche.
\newblock Approximation and optimality necessary conditions in relaxed
  stochastic control problems.
\newblock {\em Journal of Applied Mathematics and Stochastic Analysis}, 2006,
  2006.

\bibitem{balachandran1989}
Krishnan Balachandran.
\newblock Existence of optimal control for non-linear multiple-delay systems.
\newblock {\em International Journal of Control}, 49(3):769--775, 1989.

\bibitem{bender2012}
Christian Bender and Jessica Steiner.
\newblock {\em Least-squares monte carlo for backward sdes}.
\newblock Springer, 2012.

\bibitem{biagini2018}
Francesca Biagini, Thilo Meyer-Brandis, Bernt Øksendal, and Krzysztof Paczka.
\newblock Optimal control with delayed information flow of systems driven by
  \text{G}-\text{B}rownian motion.
\newblock {\em Probability, Uncertainty and Quantitative Risk}, 3(1):1--24,
  2018.

\bibitem{denis2011}
Laurent Denis, Mingshang Hu, and Shige Peng.
\newblock Function spaces and capacity related to a sublinear expectation:
  application to \text{G}-\text{B}rownian motion paths.
\newblock {\em Potential analysis}, 34:139--161, 2011.

\bibitem{denis2006}
Laurent Denis and Claude Martini.
\newblock A theoretical framework for the pricing of contingent claims in the
  presence of model uncertainty.
\newblock {\em The Annals of Applied Probability}, 16(2):827--852, 2006.

\bibitem{Elgroud2022}
Nabil Elgroud, Hacene Boutabia, Amel Redjil, and Omar Kebiri.
\newblock Existence of relaxed optimal control for \text{G}-neutral stochastic
  functional differential equations with uncontrolled diffusion.
\newblock {\em Bulletin of the Institute of Mathematics Academia Sinica},
  17(2):143--172, 2022.

\bibitem{Elsanosi2000}
Ismail Elsanosi, Bernt {\O}ksendal, and Agnes Sulem.
\newblock Some solvable stochastic control problems with delay.
\newblock {\em Stochastics: An International Journal of Probability and
  Stochastic Processes}, 71(1-2):69--89, 2000.

\bibitem{Fei2019}
Chen Fei, Wei-yin Fei, and Li-tan Yan.
\newblock Existence and stability of solutions to highly nonlinear stochastic
  differential delay equations driven by \text{G}-\text{B}rownian motion.
\newblock {\em Applied Mathematics-A Journal of Chinese Universities},
  34(2):184--204, 2019.

\bibitem{fleming1984}
Wendell~H Fleming and Makiko Nisio.
\newblock On stochastic relaxed control for partially observed diffusions.
\newblock {\em Nagoya Mathematical Journal}, 93:71--108, 1984.

\bibitem{fuhrman2010}
Marco Fuhrman, Federica Masiero, and Gianmario Tessitore.
\newblock Stochastic equations with delay: Optimal control via bsdes and
  regular solutions of hamilton--jacobi--bellman equations.
\newblock {\em SIAM Journal on Control and Optimization}, 48(7):4624--4651,
  2010.

\bibitem{gao2009}
Fuqing Gao.
\newblock Pathwise properties and homeomorphic flows for stochastic
  differential equations driven by \text{G}-\text{B}rownian motion.
\newblock {\em Stochastic Processes and their Applications},
  119(10):3356--3382, 2009.

\bibitem{Hartmann2013}
Carsten Hartmann, Ralf Banisch, Marco Sarich, Tomasz Badowski, and Christof
  Sch{\"u}tte.
\newblock Characterization of rare events in molecular dynamics.
\newblock {\em Entropy}, 16(1):350--376, 2013.

\bibitem{hu2014}
Mingshang Hu, Shaolin Ji, and Shuzhen Yang.
\newblock A stochastic recursive optimal control problem under the
  \text{G}-expectation framework.
\newblock {\em Applied Mathematics \& Optimization}, 70(2):253--278, 2014.

\bibitem{hu2018}
Mingshang Hu and Falei Wang.
\newblock Stochastic optimal control problem with infinite horizon driven by
  \text{G}-\text{B}rownian motion.
\newblock {\em ESAIM: Control, Optimisation and Calculus of Variations},
  24(2):873--899, 2018.

\bibitem{Ivanov2008}
Anatoli~F Ivanov and Anatoly~V Swishchuk.
\newblock Optimal control of stochastic differential delay equations with
  application in economics.
\newblock {\em International Journal of Qualitative Theory of Differential
  Equations and Applications}, 2(2):201--213, 2008.

\bibitem{kebiri2018}
Omar Kebiri, Lara Neureither, and Carsten Hartmann.
\newblock Singularly perturbed forward-backward stochastic differential
  equations: application to the optimal control of bilinear systems.
\newblock {\em Computation}, 6(3):41, 2018.

\bibitem{kebiri2019}
Omar Kebiri, Lara Neureither, and Carsten Hartmann.
\newblock Adaptive importance sampling with forward-backward stochastic
  differential equations.
\newblock In {\em Stochastic Dynamics Out of Equilibrium: Institut Henri
  Poincar{\'e}, Paris, France, 2017}, pages 265--281. Springer, 2019.

\bibitem{mansour2022}
Mahmoud~BA Mansour and Asmaa~H Abobakr.
\newblock Stochastic differential equation models for tumor population growth.
\newblock {\em Chaos, Solitons \& Fractals}, 164:112738, 2022.

\bibitem{Menoukeu2015}
Olivier Menoukeu~Pamen.
\newblock Optimal control for stochastic delay systems under model uncertainty:
  a stochastic differential game approach.
\newblock {\em Journal of Optimization Theory and Applications}, 167:998--1031,
  2015.

\bibitem{oksendal2011}
Bernt {\O}ksendal, Agnes Sulem, and Tusheng Zhang.
\newblock Optimal control of stochastic delay equations and time-advanced
  backward stochastic differential equations.
\newblock {\em Advances in Applied Probability}, 43(2):572--596, 2011.

\bibitem{patel1982}
NK~Patel, PC~Das, and SS~Prabhu.
\newblock Optimal control of systems described by delay differential equations.
\newblock {\em International Journal of Control}, 36(2):303--311, 1982.

\bibitem{peng2007}
Shige Peng.
\newblock \text{G}-expectation, \text{G}-\text{B}rownian motion and related
  stochastic calculus of it{\^o} type.
\newblock In {\em Stochastic Analysis and Applications: The Abel Symposium
  2005}, pages 541--567. Springer, 2007.

\bibitem{peng2010}
Shige Peng.
\newblock Nonlinear expectations and stochastic calculus under uncertainty.
\newblock {\em arXiv preprint arXiv:1002.4546}, 24, 2010.

\bibitem{redjil2018}
Amel Redjil and Salah~Eddine Choutri.
\newblock On relaxed stochastic optimal control for stochastic differential
  equations driven by \text{G}-\text{B}rownian motion.
\newblock {\em ALEA, Lat.Am. J. Probab. Math. Stat}, 15:201–212, 2018.

\bibitem{Ren2017}
Yong Ren, Jun Wang, and Lanying Hu.
\newblock Multi-valued stochastic differential equations driven by
  \text{G}-\text{B}rownian motion and related stochastic control problems.
\newblock {\em International Journal of Control}, 90(5):1132--1154, 2017.

\bibitem{Ren2018}
Yong Ren, Wensheng Yin, and Rathinasamy Sakthivel.
\newblock Stabilization of stochastic differential equations driven by
  \text{G}-\text{B}rownian motion with feedback control based on discrete-time
  state observation.
\newblock {\em Automatica}, 95:146--151, 2018.

\bibitem{rosenblueth1992}
Javier~F Rosenblueth.
\newblock Strongly and weakly relaxed controls for time delay systems.
\newblock {\em SIAM journal on control and optimization}, 30(4):856--866, 1992.

\bibitem{soner2011m}
H~Mete Soner, Nizar Touzi, and Jianfeng Zhang.
\newblock Martingale representation theorem for the \text{G}-expectation.
\newblock {\em Stochastic Processes and their Applications}, 121(2):265--287,
  2011.

\bibitem{soner2011}
Mete Soner, Nizar Touzi, and Jianfeng Zhang.
\newblock Quasi-sure stochastic analysis through aggregation.
\newblock {\em Electronic Journal of Probability}, 16:1844--1879, 2011.

\bibitem{stoica2004}
George Stoica.
\newblock A stochastic delay financial model.
\newblock {\em Proceedings of the American Mathematical Society},
  133(6):1837--1841, 2004.

\bibitem{xu2010}
Yuhong Xu.
\newblock Backward stochastic differential equations under super linear
  \text{G}-expectation and associated hamilton-jacobi-bellman equations.
\newblock {\em arXiv preprint arXiv:1009.1042}, 2010.

\bibitem{yang2016}
Jie Yang and Weidong Zhao.
\newblock Numerical simulations for \text{G}-\text{B}rownian motion.
\newblock {\em Frontiers of Mathematics in China}, 11:1625--1643, 2016.

\bibitem{Yin2022}
Wensheng Yin, Jinde Cao, and Guoqiang Zheng.
\newblock Further results on stabilization of stochastic differential equations
  with delayed feedback control under $ g $-expectation framework.
\newblock {\em Discrete and Continuous Dynamical Systems-B}, 27(2):883--901,
  2022.

\bibitem{Yua2021}
Haiyan Yuan.
\newblock Some properties of numerical solutions for semilinear stochastic
  delay differential equations driven by \text{G}-\text{B}rownian motion.
\newblock {\em Mathematical Problems in Engineering}, 2021:1--26, 2021.

\end{thebibliography}

\end{document}